\documentclass[reqno,10pt]{amsart}

\usepackage{amsmath, amsthm, amsfonts}
\usepackage{pdfsync}
\usepackage{parskip}
\usepackage{hyperref}
\usepackage{mathrsfs}  
\usepackage{mathtools}
\usepackage{mdframed}
\usepackage{tikz}
\numberwithin{equation}{section}
\newtheorem{theorem}{Theorem}[section]
\newtheorem{prop}[theorem]{Proposition}
\newtheorem{definition}[theorem]{Definition}
\newtheorem{lem}[theorem]{Lemma}
\newtheorem{cor}[theorem]{Corollary}

\theoremstyle{remark}
\newtheorem{remark}[theorem]{Remark}

\def\cA{\mathcal{A}}
\def\cE{\mathcal{E}}
\def\cH{\mathcal{H}}
\def\cL{\mathcal{L}}
\def\cM{\mathcal{M}}
\def\M{\mathsf{M}}

\def\R{\mathbb{R}}
\def\bC{\mathbb{C}}
\def\bN{\mathbb{N}}
\def\Nz{\mathbb{N}_0}
\def\Lis{\mathcal{L}{\rm{is}}}

\def\B{\mathbb{B}}

\def\id{{\rm{id}}}
\def\min{{\rm min}}
\def\max{{\rm max}}

\def\FD{\underline{\Delta}_F}

\begin{document}

\title[The fractional porous medium equation on conical manifolds II]{The fractional porous medium equation on manifolds with conical singularities II}

\author[N. Roidos]{Nikolaos Roidos}
\address{Department of Mathematics, University of Patras, 26504 Rio Patras, Greece}
\email{roidos@math.upatras.gr}

\author[Y. Shao]{Yuanzhen Shao}
\address{Department of Mathematics,
         The University of Alabama,
         Box 870350,
         Tuscaloosa, AL 35487-0350,
         USA}
\email{yshao8@ua.edu}


\begin{abstract}
This is the second of a series of two papers which studies the fractional porous medium equation, $\partial_t u +(-\Delta)^\sigma (|u|^{m-1}u  )=0 $ with $m>0$ and $\sigma\in (0,1]$, posed on a Riemannian manifold with isolated conical singularities. 
The first aim of the article is to derive some useful properties for the Mellin-Sobolev spaces including the Rellich-Kondrachov  Theorem and Sobolev-Poincar\'e, Nash and Super Poincar\'e  type inequalities.  
The second part of the article is devoted to the study the Markovian extensions of the conical Laplacian operator and its fractional powers.
Then based on the obtained results, 
we establish existence and uniqueness of a global strong solution for $L_\infty-$initial data and all $m>0$.
We further investigate a number of properties of the solutions, including comparison principle, $L_p-$contraction and conservation of mass.
Our approach is quite general and thus is applicable to a variety of similar problems on manifolds with more general singularities.
\end{abstract}

\maketitle

\section{Introduction}\label{Section 1}

The study of nonlocal diffusion equations has become an active research field, both from the point of view of mathematical theory and in relation with its wide applications to real-world problems in material sciences, finance, biology and so forth, c.f. \cite{Agui, Bucur, Carreras, Janev}.
The fractional porous medium equation, arguably, is one of the most extensively studied nonlinear nonlocal diffusion equations. 
There is a vast amount of work on fractional porous medium equations in Euclidean spaces, cf. \cite{AkaSchSeg16, AlpEll16, BonSireVaz15, BonVaz15, BonVaz16, FanLiZhang, PunTer14, PalRodVaz11, PalRodVaz12}, just to name a few.
 
In this paper, we consider the fractional porous medium equation:
\begin{equation}
\label{S1: FPME}
\left\{\begin{aligned}
\partial_t u +(-\Delta_g)^\sigma (|u|^{m-1}u  )&=0   &&\text{on}&&\M\times (0,\infty);\\
u(0)&=u_0   &&\text{on}&&\M
\end{aligned}\right.
\end{equation}
for $m>0$ and $\sigma\in (0,1]$, on an $(n+1)$--dimensional Riemannian manifolds $(\M,g)$ with isolated conical singularities. 
Here $\Delta_g$ is the Laplace-Beltrami operator associated with the conical metric $g$, called the conical Laplacian.
A conical manifold, in the simplest case, is  a  product $\M=(0,\varepsilon)\times N$ with a smooth and compact manifold $(N,g_N)$ near an conical singularity, equipped with the degenerate metric 
$
g(x,y)=dx^2 +x^2 g_N(y)
$
for $(x,y)\in (0,\varepsilon)\times N$.

The objective of our work is threefold. 

First, the theory of Mellin-Sobolev spaces is widely used in the study of differential equations on conical manifolds. Briefly speaking, Mellin-Sobolev spaces are weighted Sobolev spaces with weights defined in terms of $x$. In this article, we establish several interesting properties like Rellich-Kondrachov Theorem, Sobolev-Poincar\'e inequality, Nash inequality and Super Poincar\'e inequalities for Mellin-Sobolev spaces with various weights. Some of these properties are already present for some limited cases in   previous works. However, they have never been established in this generality. 
In particular, the derived properties can be used  to investigate the asymptotic behavior of the solution to \eqref{S1: FPME} on conical manifolds. To keep the paper in a reasonable length, we will explore this topic in a subsequent paper.

Second, we conduct a careful study of an important class of closed extensions, the Markovian extension, c.f. Section~\ref{Section 3.2}, of the conical Laplacian operator $\Delta_g$ and the Markovian property of the fractional conical Laplacian.
A major difficulty in the study of the conical Laplacian operator is  that it does not have a  canonical choice for a closed extension, as noticed  by J. Br\"uning and R. Seeley \cite{BruSee88}.  See also \cite{Cheeger79, Cheeger83, Cheeger87}.
Indeed,  the domains of the minimal and the maximal closed extensions of $\Delta_g$ differ by a non-trivial  finite dimensional space. 
Functions in this finite dimensional space admit certain asymptotic behaviors as $x\to 0^+$, as we will see in Section~\ref{Section 3.1}.
Interested readers may find more details in \cite{CorSchSei03, Les97, RoiSch14, RoiSch15,SchSei05}.

Among all possible closed extensions of a differential operator, the Markovian extension plays an important role in the study of partial differential equations.
For example, once we find a  Markovian extension of a  differential operator in $L_2(\M)$,  then the extension or restriction of the operator onto $L_p(\M)$ for any $p\in (1,\infty)$ generates a positive contraction analytic semigroup.
Following a  transference method argument, see for instance \cite{Shao18}, we can show that the operator has  maximal $L_p-$regularity property.	
It is well-known that, combining with appropriate techniques, maximal $L_p-$regularity property  can be used to study the well-posedness, regularity and stability of differential equations or systems.
The monograph \cite{PruSim16} by J.~Pr\"uss and G.~Simonett presents an in-depth discussion of maximal $L_p-$regularity theory and its applications.
Despite of its importance,  very little is known about the Markovian extension of differential operators on conical manifolds.
Indeed, to the best of our knowledge, even for the conical Laplacian, many properties like the uniqueness of  Markovian extension  are only known for the case ${\rm dim}(\M)=2$.
In this article,  we carefully study  the existence and uniqueness of Markovian extension of the conical Laplacian, cf. Propositions~\ref{S4: Delta-positivity} and  \ref{Prop: unique Markov}, in arbitrary dimensions.
Denote  by  $\FD$ the unique Markovian extension.
Then the dissipativity of the extensions or restriction of $-\FD$ to $L_p(\M)$ follows from a classic argument by E. Davies \cite{Dav89}.
Based on these results, we are able   to construct the fractional  conical Laplacian operators  and show that the fractional operator is again Markovian. 

The third and most important objective of this article is to find a suitable approach to \eqref{S1: FPME}, which is capable of dealing with both possibly degenerate initial data and the presence of conical singularities.
In Euclidean space $\R^N$, to the best of our knowledge, both the Cauchy and Dirichlet problems of 
\eqref{S1: FPME} were studied via one of the following approaches or their analogues.

(1) In \cite{PalRodVaz11, PalRodVaz12}, the authors constructed the Fractional Laplacian via the Caffarelli-Silvestre extension \cite{CafSil07}, i.e. consider the problem:
$$
\nabla \cdot (y^{1-2\sigma} \nabla w)=0 \quad  (x,y)\in \R^N\times \R_+; \quad w(x,0)=u(x),\quad x\in \R^N.
$$
For some constant $C_\sigma$, it holds that
$$
(-\Delta)^\sigma u(x)= -C_\sigma \lim\limits_{y\to 0} y^{1-2\sigma}\frac{\partial w}{\partial y}(x,y).
$$
Then \eqref{S1: FPME} was transformed into a quasi-stationary problem  with a dynamical boundary condition on $\R^N \times \R_+$, and solutions were obtained by means of Crandall-Liggett's Theorem.
See also \cite{AlpEll16, PunTer14}.

(2) In   \cite{BonSireVaz15, BonVaz15, BonVaz16}, M. Bonforte, Y. Sire and J.L. V\'azquez used the Spectral Fractional Laplacian (SFL) in a bounded domain $\Omega$ defined by:
$$
(-\Delta)^\sigma u (x):= \frac{1}{\Gamma(-\sigma)}\int_0^\infty (e^{t\Delta}u(x)-u(x))\,\frac{dt}{t^{1+\sigma}}= \sum\limits_{j=1}^\infty \lambda_j^\sigma \hat{u}_j \phi_j(x),
$$
where $(\phi_k, \lambda_k)_{k=1}^\infty$ denote an orthonormal basis of
$L_2(\Omega)$ consisting of eigenfunctions of the Dirichlet Laplacian $-\Delta$ in $\Omega$  and their corresponding eigenvalues.
$\hat{u}_k$ are the Fourier coefficients of $u$.
A solution  was obtained via monotone operators techniques in Hilbert spaces or approximation by mild solutions.
This approach  relies on either good estimates of the Green functions or certain growth condition of $\lambda_k$. Unfortunately, both are unknown on general conical manifolds.

(3) Another way to construct the Fractional Laplacian is by using the integral representation in terms of hypersingular kernels:
$$
(-\Delta)^\sigma u (x)= C_{N,\sigma} P.V.\int_{\R^N} \frac{u(x) - u(y)}{|x+y|^{N+2\sigma}}\, dy.
$$
It is understood that $u$ is extended by zero outside $\Omega$.
This is called Restricted Fractional Laplacian (RFL). 
Approaches based on RFL are similar to those in (2).

The presence of conical singularities in the underlying space $\M$ changes Equation~\eqref{S1: FPME} in a fundamental way.
Indeed, it is quite obvious that (RFL) is not applicable to non-Euclidean spaces in general. 
For the reason stated above, (SFL) also seems to be restrictive in the case of a conical manifold.
In the Caffarelli-Silvestre extension approach, a crucial component   is the trace and extension theorems between a weighted space defined on $\R^N \times \R_+$ and $\dot{H}^{\sigma}(\R^N)$. However, if the underlying space $\M$ is a conical manifold, $\M\times \R_+$ becomes a corner manifold and such trace and extension theorems do not seem to be true in this case.

In \cite{RoidosShao}, we initiated the study of \eqref{S1: FPME} on conical manifolds, where we proved that operators of the form $w (-\Delta_g)^\sigma$ enjoys maximal $L_p-$regularity property for some proper strictly positive continuous function $w$. 
Then combining with a fixed point theorem argument, we established the existence and uniqueness of classical solution 	for strictly positive H\"older continuous data.
To deal with general bounded initial data, in this article, we take a completely different approach to \eqref{S1: FPME}.
Our approach only relies on the existence of a Markovian extension of $\Delta_g 	$.
Based on the Markovian properties of $(-\FD)^\sigma$ proved in Section~\ref{Section 4}, we are able to establish  the following for \eqref{S1: FPME}.
\begin{theorem}
\label{S1: Main theorem}
Suppose that $(\M,g)$ is an $(n+1)$-dimensional conical manifold. 
Then for any  $ u_0\in  L_\infty(\M)$, \eqref{S1: FPME} with $\sigma\in (0,1)$ and $m>0$ has a unique global strong solution $u$. Moreover, the solution $u$ satisfies the following properties.
\begin{enumerate}
\item  {\em Continuous dependence on the  initial data:} The solution depends continuously on $u_0$ in the norm $C([0,T), L_1(\M))$.
\item  {\em Comparison principle:} If $u, \hat{u}$ are the unique strong solutions to \eqref{S8: PME}  with initial data $u_0, \hat{u}_0$, respectively, then $u_0\leq \hat{u}_0$ a.e. implies $u \leq \hat{u}$ a.e.
\item {\em $L_p$-contraction:} For all $0 \leq t_1 \leq t_2$ and $1\leq p\leq \infty$
$$
\|u(t_2)\|_p \leq \|u(t_1)\|_p.
$$
\item  {\em Conservation of mass:} For  any $t\geq 0$, it holds   that
$$
\int_\M u(t) \, d\mu_g = \int_\M u_0 \, d\mu_g.
$$
\end{enumerate}
\end{theorem}
We would like to refer the reader to Definitions~\ref{Def: weak sol} and \ref{Def: Strong Sol} for the definitions of solutions used in the paper.
Analogous results also hold when $\sigma=1$, i.e. the usual porous medium equation, see Section~\ref{Section 8}.

It is worthwhile mentioning that the method in this article can be applied to similar problems on manifolds with more general singularities, including those with cuspidal, edge and corner singularities. We will discuss this idea in a future work.
The method in this paper also seems to give an alternative approach to the Cauchy and Dirichlet problem of the fractional porous medium equation in $\R^N$.


The paper is organized as follows:

In Section~\ref{Section 2}, we give the precise  definitions of conical manifolds and   the Mellin-Sobolev spaces. Then we   prove a  Rellich-Kondrachov type theorem for Mellin-Sobolev spaces.

In Section~\ref{Section 3.1}, we give a brief review of the closed extensions of the conical Laplacian operator $\Delta_g$. 
In Section~\ref{Section 3.2}, we study the Markovian property of the Friedrichs extension of $\Delta_g$ and fix the closed extensions of $\Delta_g$ that we used in the analysis of \eqref{S1: FPME}. These closed extensions are denoted by $\underline{\Delta}_{F,p}$ for $1\leq p <\infty$.
In Section~\ref{Section 3.3}, we prove that $\Delta_g$ indeed admits only one Markovian extension.
Section~\ref{Section 3.4} is where we establish the Sobolev-Poincar\'e, Nash and Super Poincar\'e inequalities for Mellin-Sobolev spaces.

In Section~\ref{Section 4.1}, we construct the fractional powers of $\omega-\underline{\Delta}_{F,p}$ for all $\omega\geq 0$; and then in Section~\ref{Section 4.2}, we derive the crucial properties of  these operators. 
These results form the theoretic basis for the study of \eqref{S1: FPME}.

In Section~\ref{Section 5}, we employ the results from Sections 2-4 and prove that \eqref{S1: FPME} admits a unique global weak solution depending continuously on the initial data. In Section~\ref{Section 6}, we further investigate various properties of weak solutions ((2)-(4) of Theorem~\ref{S1: Main theorem}).
In Section~\ref{Section 7}, we prove that   \eqref{S1: FPME} indeed has a unique global strong solution.

In Section~\ref{Section 8}, we include a discussion for the porous medium equation, i.e. $\sigma=1$ in \eqref{S1: FPME}, and show that results parallel to those in Theorem~\ref{S1: Main theorem} can be established.

\textbf{Notations:} 

For any two Banach spaces $X,Y$, $X\doteq Y$ means that they are equal in the sense of equivalent norms. The notations
$$
X\hookrightarrow Y, \qquad X\xhookrightarrow{d} Y, \qquad X\xhookrightarrow{c} Y
$$
mean that $X$ is continuously embedded, densely embedded and compactly embedded into $Y$, respectively.
$\cL(X,Y)$ denotes the set of all bounded linear maps from $X$ to $Y$, and $\cL(X):=\cL(X,X)$. Moreover, $\Lis(X,Y)$ stands for the subset of $\cL(X,Y)$ consisting of all bounded linear isomorphisms from $X$ to $Y$. 
Given a sequence $(u_k)_k:=(u_1,u_2,\cdots)$ in $X$, $u_k \rightharpoonup u$ in $X$ means that $u_k$ converge weakly to some $u\in X$.
Given a densely-defined operator $\cA$ in $X$, $D(\cA)$ stands for the domain of $\cA$.


In addition, $\R_+:=[0,\infty)$ and $\Nz:=\bN\cup \{0\}$.




\section{Preliminaries}\label{Section 2}

\subsection{Conical Manifolds}\label{Section 2.1}

We can construct an $(n+1)$-dimensional {\em conical manifold} $\M$  from a $C^\infty$-compact manifold $(\tilde{\M}, \tilde{g})$ with 
possibly disconnected smooth boundary. Let $(\B,h)=(\partial\tilde{\M}, \tilde{g}_{\partial\tilde{\M}})$.
We may assume that $\B=\bigsqcup_{i=1}^{k_\B} \B_i$ for some $k_\B\in \bN$, where $\B_i$ are closed, smooth and connected.
We equip $\M=\tilde{\M}\setminus \B$ with a smooth metric $g=\{g_{ij}\}_{i,j\in\{0,1,\cdots,n\}}$ such that in local coordinates 
$$
(y_0,y_1,\cdots,y_n)=(x,y_1,\cdots,y_n)=(x,y)
$$ 
of a closed collar neighbourhood $(0,1]\times \B$ of the boundary $(\B,h)$
$$
g(x,y)= dx^2 + x^2 h(y),\quad (x,y)\in (0,1]\times\B.
$$
Here $x$ is a boundary defining function of $\tilde{\M}$ and $h=\{h_{ij}\}_{j,j\in\{1,\cdots,n\}}$. Outside $(0,1]\times \B$, $g$ is equivalent to $\tilde{g}$  in the sense that there is a constant $c\geq 1$ such that
$$
\displaystyle  (1/c)\tilde{g}(\xi,\xi)(\mathsf{p}) \leq g(\xi,\xi) \leq c\tilde{g}(\xi,\xi)(\mathsf{p}),\quad \xi\in T_{\mathsf{p}} \M , \, \, \mathsf{p}\in \M.
$$  

\begin{remark}
\begin{itemize}
\item[]
\item[1. ] In this article, we will only consider the non-trivial case $n\geq 1$.
\item[2. ] Under the compactness assumption of $(\tilde{\M}, \tilde{g})$, we have
$L_q(\M) \hookrightarrow L_p(\M) $ for $q>p$, since ${\rm Vol}(\M)<\infty$. 
We pose the compactness assumption mainly for the sake of simplicity. In a subsequent paper \cite{RoidosShaoPre1}, we will show with slight modification the method in this article can be applied  to non-compact case as well.
\end{itemize}
\end{remark}


\subsection{Mellin-Sobolev spaces}\label{Section 2.2}
Here we will describe a scale of weighted Sobolev spaces $\cH^{s,\gamma}_p (\M)$. These weighted Sobolev spaces are widely used in the analysis on conical manifolds. See \cite{Les97, RoiSch14, RoiSch15}.

Let $I:=(0,1]$.
We pick a cut-off function $\psi$ on $I$, which means $\psi\in C^\infty(I,[0,1])$ with $\psi\equiv 1$ near $0$ and $\psi\equiv 0$ near $1$.

For $k\in\Nz$,  $\cH^{k,\gamma}_p(\M)$ is the space of all functions $u\in H^k_{p,loc}(\M)$ such that near the conical singularities, or more precisely, in $I\times \B$
$$
x^{\frac{n+1}{2}-\gamma}(x \partial_x)^j \partial^\alpha_y (\psi u)  \in L_p( I \times \B, \frac{dx}{x}dy),\quad j+|\alpha|\leq k, \,\, \alpha\in \bN^n_0,
$$
where $(x,y)\in I\times \B$.
Here $\partial_y^\alpha$ can be considered as the derivatives in local coordinates of $\B$, and we will use this slight abuse of notation throughout this paper.


To understand the motivation of this somewhat unusual definition, let us consider the flat cone $\M= I\times S^n$ in $\R^{n+1}$, where $S^n$ is the $n$-sphere. Taking polar coordinates in $\M$, then $\cH^{0,0}_2(\M)$ coincides with the usual $L_2(B_1)$ space ($B_1$ is the closed unit ball in $\R^{n+1}$). 

For arbitrary $\gamma\in\R$, we define
$$
S_\gamma: C_c^\infty(I\times \B) \to C_c^\infty(\R_+\times \B): \, u(x,y)\mapsto e^{(\gamma-\frac{n+1}{2})t} u(e^{-t},y).
$$
Then for any $s\geq 0$ and $\gamma\in\R$,  $\cH^{s,\gamma}_p(\M)$ is the space of all distributions on $\M$ such that
$$
\|u\|_{\cH^{s,\gamma}_p(\M)}=\|S_\gamma ( \psi u )\|_{H^s_p(\R_+\times\B)} +  \| (1- \psi) u \|_{H^s_p(\M)} <\infty.
$$
It is understood that $\psi$ is extended to be zero outside $I \times \B$.




\begin{lem}
\label{S2.1: lem-isometry}
For all $s\geq 0$, $
S_\gamma \in \Lis(\cH^{s,\gamma}_p(I\times \B), H^s_p(\R_+ \times \B)).
$
\end{lem}
\begin{proof}
This easily follows from the definition of Mellin-Sobolev spaces.
\end{proof}

It is immediate from the definition of Mellin-Sobolev spaces that 
\begin{equation}
\label{Sobolev embedding}
\cH^{s_1,\gamma_1}_p(\M)\hookrightarrow \cH^{s_0,\gamma_0}_q(\M).
\end{equation}
if (1) $s_1 -\frac{n+1}{p} =  s_0- \frac{n+1}{q}$, $s_1 \geq  s_0$ and $\gamma_1 \geq \gamma_0$ or (2) $s_1 -\frac{n+1}{p} > s_0- \frac{n+1}{q}$, $s_1 \geq  s_0$ and $\gamma_1 > \gamma_0$.

We will prove a Rellich-Kondrachov type theorem on conical manifolds. See also \cite{Fedosov} and \cite[Remark~2.1(b)]{Seiler01}.
\begin{prop}
\label{S2.1: Rellich-thm}
Assume that $s_1 -\frac{n+1}{p}> s_0- \frac{n+1}{q}$, $s_1>s_0$ and $\gamma_1> \gamma_0$. Then 
$$
\cH^{s_1,\gamma_1}_p(\M)\xhookrightarrow{c} \cH^{s_0,\gamma_0}_q(\M).
$$
\end{prop}
\begin{proof}
Take $\gamma_1>\gamma>  \gamma_0$ and assume that $(u_k)_k\subset \cH^{s_1,\gamma_1}_p(\M)$ with
$$
\|u_k\|_{\cH^{s_1,\gamma_1}_p} \leq M.
$$
Since the Rellich-Kondrachov theorem holds for compact manifolds with smooth boundary, by using  a partition of unity, we may assume that $(u_k)_k$ are supported in $I\times \B$. 

Lemma~\ref{S2.1: lem-isometry} immediately implies that $\|S_{\gamma_1} (u_k)\|_{s_1,p}\leq M$. Put $\B_i=[0,i]\times\B$. It is understood that $\B_0=\emptyset$. 
$\B_i$ are compact manifolds with smooth boundaries.
Then for each $i$, there exists a subsequence, still denoted by $(u_k)_k$, such that
$$
S_{\gamma_1}( u_k)|_{\B_i} \to v_i \quad \text{in}\quad H^{s_0}_q (\B_i).
$$
Moreover, for $j>i$, we have $v_j|_{\B_i}=v_i$. Hence we obtain a function $v$ on $\R_+\times \B$ such that $v|_{\B_i}=v_i$. 
Pick  $s_1 \leq k\in \bN$.
By the standard point-wise multiplication theorem, we can find a constant $M_1$ independent of $i$ such that
$$
\| w_1 w_2 \|_{H^{s_0}_q (\B_i\setminus \B_{i-1})} \leq M_1 \|w_1\|_{BC^k(\B_i\setminus \B_{i-1})} \|w_2\|_{H^{s_0}_q (\B_i\setminus \B_{i-1})} 
$$
for all $w_1 \in BC^k(\B_i\setminus \B_{i-1})$ and $w_2\in H^{s_0}_q (\B_i\setminus \B_{i-1})$,
where the space $BC^k(\B_i\setminus \B_{i-1})$ consists of all functions $w\in C^k(\B_i\setminus \B_{i-1})$ such that 
$$
\|w\|_{BC^k(\B_i\setminus \B_{i-1})}:=\sum\limits_{j+l\leq k} \|\partial_t^j \nabla_h^l w\|_\infty<\infty.
$$
Here $\nabla_h$ is the  covariant derivative with respect to $h$.
Let $M_2:=\sum\limits_{j=0}^k (\gamma_1-\gamma_0)^j$.
One can then compute that
\begin{align*}
\|e^{(\gamma -\gamma_1)t} v(t,\cdot)\|_{H^{s_0}_q (\R_+ \times \B)} &= \sum\limits_{i=1}^\infty \|e^{(\gamma -\gamma_1)t} v(t,\cdot)\|_{H^{s_0}_q (\B_i\setminus \B_{i-1})}\\
& \leq M_1  \sum\limits_{i=1}^\infty \| e^{(\gamma -\gamma_1)t}\|_{BC^k (\B_i\setminus \B_{i-1})} \| v(t,\cdot)\|_{H^{s_0}_q (\B_i\setminus \B_{i-1})}\\
& \leq M_1 M_2 \sum\limits_{i=1}^\infty   e^{(\gamma -\gamma_1)(i-1)}   \| v(t,\cdot)\|_{H^{s_0}_q (\B_i\setminus \B_{i-1})}\\
& \leq   M   M_1 M_2   \sum\limits_{i=1}^\infty e^{(\gamma -\gamma_1)(i-1)}  =: M^\prime.
\end{align*}
Hence $e^{(\gamma -\gamma_1)t} v(t,\cdot)\in H^{s_0}_q(\R_+\times \B)  $.
For any $\varepsilon>0$, take $i$ large enough such that  
$$
M_1 M_2 e^{(\gamma_0-\gamma )i} (M+M^\prime)<\varepsilon/2.
$$
Then letting $\B_i^c=(\R_+\times \B ) \setminus \B_i$, we have
\begin{align*}
&\|S_{\gamma_0}(u_k) - e^{(\gamma_0-\gamma_1)t}v(t,\cdot)\|_{H^{s_0}_q (\R_+\times \B)}\\
 \leq & \|S_{\gamma_0}(u_k) - e^{(\gamma_0-\gamma_1)t}v(t,\cdot)\|_{H^{s_0}_q(\B_i)}
 + \|S_{\gamma_0}(u_k) - e^{(\gamma_0-\gamma_1)t}v(t,\cdot)\|_{H^{s_0}_q(\B_i^c)}\\
 \leq & \| e^{(\gamma_0-\gamma_1)t} (S_{\gamma_1}(u_k)- v_i)\|_{H^{s_0}_q(\B_i)} \\
 &+ M_1 \sum\limits_{  j=i+1 }^\infty \| e^{(\gamma_0-\gamma)t}\|_{BC^k(\B_j \setminus \B_{j-1})} \| S_\gamma(u_k)-e^{(\gamma-\gamma_1)t} v(t,\cdot)\|_{H^{s_0}_q (\B_j\setminus \B_{j-1})}\\
 \leq & \| e^{(\gamma_0-\gamma_1)t} (S_{\gamma_1}(u_k)- v_i)\|_{H^{s_0}_q(\B_i)} +\varepsilon/2 \leq \varepsilon
\end{align*}
for $k$  large enough. This implies that
$$
S_{\gamma_0}(u_k) \to e^{(\gamma_0-\gamma_1)t}v(t,\cdot) \quad \text{in} \quad H^{s_0}_q(\R_+\times \B).
$$
Now it follows from Lemma~\ref{S2.1: lem-isometry} that
$$
\| u_k - S_{\gamma_0}^{-1}(e^{(\gamma_0-\gamma_1)t}v(t,\cdot)) \|_{\cH^{s_0,\gamma_0}_q} \to 0 \quad \text{as}\quad k\to \infty.
$$
This establishes the compact embedding assertion.
\end{proof}


For $0<\theta <1$ and $1\leq q \leq \infty$, we denote by   $(\cdot,\cdot)_{\theta,q}$ the real interpolation method, cf. \cite[Example I.2.4.1]{Ama95}.


\begin{definition}
Recall that $\B=\sqcup_{i=1}^{k_\B} \B_i$.
Denote by $\bC_\psi$ the space of all  smooth  functions $c$ that vanish on $\M\setminus(I\times\B)$ and on each component $I\times\B_i$, $i\in\{1,...,k_\B\}$, they are of the form $c_i\psi$, where $c_i \in \bC$, i.e. $\bC_\psi$ consists of smooth functions that are locally constant close to the boundary. Endow $\bC_\psi$ with the norm $\|\cdot\|_{\bC_\psi}$ given by $c\mapsto \|c\|_{\bC_\psi}:=(\sum_{i=1}^{k_\B}|c_i|^2)^{\frac{1}{2}}$. 
\end{definition}

\begin{lem}
\label{S2.1: Sobolev-interpolation}
Suppose that $0\leq s_0<s_1$ and $\gamma_0,\gamma_1\in\R$.
For any $ \varepsilon>0$ and $0<\theta<1$,
$$
\cH^{s+\varepsilon,\gamma+\varepsilon}_p(\M) \hookrightarrow  (\cH^{s_0,\gamma_0}_p(\M), \cH^{s_1,\gamma_1}_p(\M) )_{\theta,p}   \hookrightarrow \cH^{s-\varepsilon,\gamma-\varepsilon}_p(\M) ,
$$
where  $s=(1-\theta)s_0+\theta s_1$ and $\gamma=(1-\theta)\gamma_0 +\theta \gamma_1$, and for any $s\geq 0$
\begin{align*}
\cH^{s+2\theta+\varepsilon,\gamma+2\theta+\varepsilon}_p(\M)+ \bC_\psi &\hookrightarrow  (\cH^{s,\gamma}_p(\M), \cH^{s+2,\gamma+2}_p(\M)+\bC_\psi)_{\theta,p}\\
   &\hookrightarrow \cH^{s+2\theta-\varepsilon,\gamma+2\theta-\varepsilon}_p(\M)+\bC_\psi.
\end{align*} 
\end{lem}
\begin{proof}
The first embeddings follow from \cite[Lemma~5.4]{CorSchSei02} and \cite[Lemma~3.6]{RoiSch15}. The second embeddings  were proved in \cite[Lemma~5.2]{RoiSch15}. 
\end{proof}


\section{The Laplace-Beltrami operator on conical manifolds}\label{Section 3}

\subsection{Closed extensions of $\Delta_g$}\label{Section 3.1}



The Laplace-Beltrami operator $\Delta_g$ induced by the conical metric $g$ is a second order differential operator, which near the conical singularities, i.e. inside $(0,1]\times \B$, can be written as
$$
\Delta_g = \frac{1}{x^2}  [ (x\partial_x)^2 +(n-1)(x\partial_x)  + \Delta_h ],
$$
where $\Delta_h$ is the Laplace-Beltrami operator on $\B$ with respect to $h$.


The conormal symbol of $\Delta_g$ is defined by
$$
\sigma_\M(\Delta_g)(z):=z^2 - (n-1)z+\Delta_\B,\quad z\in \bC.
$$
In particular, $\sigma_\M(\Delta_g)\in \cA(\bC, \cL(H^{s+2}_p(\B), H^s_p(\B)))$, where $\cA(\bC,E)$  stands for the space  of analytic $E$-valued functions on $\bC$ for any Banach space $E$.


If we consider $\Delta_g$ as an unbounded operator  in  $\cH^{s,\gamma}_p(\M)$ with domain $C^\infty_c(\M)$,  denote its closure by $\underline{\Delta}_{\min}=\underline{\Delta}_{s,\min}^\gamma$, and its maximal closed extension by $\underline{\Delta}_{\max}=\underline{\Delta}_{s,\max}^\gamma$, where
$$
D(\underline{\Delta}_{\max})=\{u\in \cH^{s,\gamma}_p(\M): \Delta_g u\in \cH^{s,\gamma}_p(\M)\}.
$$
We have
\begin{align}
\label{S3.1: min domain}
\notag D (\underline{\Delta}_{\min})&= D(\underline{\Delta}_{\max}) \cap \bigcap_{\varepsilon>0} \cH^{s+2,\gamma+2-\varepsilon}_p(\M)\\
&=\{u\in \bigcap_{\varepsilon>0} \cH^{s+2,\gamma+2-\varepsilon}_p(\M): \Delta_g u\in \cH^{s,\gamma}_p(\M)\}.
\end{align}
In particular, $D (\underline{\Delta}_{\min})=\cH^{s+2,\gamma+2}_p(\M)$ iff $\sigma_\M(\Delta_g)(z)$ is invertible  on the line 
$$
\displaystyle \{ z\in \bC: {\rm Re} z=\frac{n-3}{2}-\gamma\}.
$$
The reader may refer to \cite[Proposition~5.1]{SchSei05} for the details of this result; 
note that
the structure of the minimal domain in general stratified space is a highly non-trivial
issue, cf. \cite{Albin, Hartmann}.

We denote by $0=\lambda_0>\lambda_1>\cdots$ the distinct eigenvalues of $\Delta_\B$ and by $E_0, E_1, \cdots$ the corresponding eigenspaces.
Then the non-bijectivity points of $\sigma_\M(\Delta_g)$ are exactly
\begin{equation}
\label{S3: q_j}
q^{\pm}_j= \frac{n-1}{2} \pm \sqrt{(\frac{n-1}{2})^2 -\lambda_j},\quad j\in\Nz.
\end{equation}
When  $\displaystyle q^{\pm}_j\neq \frac{n+1}{2}-2- \gamma $ for all $j\in \Nz$,
the minimum domain of $\Delta_g$ in $\cH^{s,\gamma}_p(\M)$ is
\begin{equation}
\label{S3: min-domain}
D (\underline{\Delta}_{\min})=\cH^{s+2,2+\gamma}_p(\M).
\end{equation}
For $q^{\pm}_j$ with $j\neq 0$, we define the spaces
\begin{equation}
\label{S4: E_{qj}-1}
\cE_{q^{\pm}_j}=\psi x^{-q^{\pm}_j}\otimes E_j = \{\psi(x)x^{-q^{\pm}_j} e_j(y): e_j \in E_j \}.
\end{equation}
When $j=0$, put
\begin{equation}
\label{S4: E_{qj}-2}
\cE_{q^{\pm}_0}=
\begin{cases}
\psi x^{-q^{\pm}_0}\otimes E_0   \qquad &n>1;\\
\psi\otimes E_0 + \psi \log x\otimes E_0  &n=1.
\end{cases}
\end{equation}
We will also introduce the set $I_\gamma$ defined by
\begin{equation}
\label{S3.1: I_gamma}
I_\gamma:= (\frac{n+1}{2}-\gamma-2, \frac{n+1}{2}-\gamma).
\end{equation}
As a conclusion from \cite[Theorem~3.6]{GilKra06} and \cite[Proposition~5.1]{SchSei05},  we have
\begin{equation}
\label{S2.2: max domain}
D (\underline{\Delta}_{\max})=D (\underline{\Delta}_{\min})\oplus \bigoplus_{q^{\pm}_j\in I_\gamma} \cE_{q^{\pm}_j}.
\end{equation}


Let $\displaystyle \gamma_p=\frac{n+1}{2}-\frac{n+1}{p}$ for $1\leq p\leq \infty$. Note that 
\begin{equation}
\label{S3: equiv Lp}
\cH^{0,\gamma_p}_p(\M)=L_p(\M) \quad \text{on }(\M,g).
\end{equation}
For this, we will denote the norm $\|\cdot\|_{\cH^{0,\gamma_p}_p}$ simply by $\|\cdot\|_p$ throughout.
The above results are obtained from the theory of Schulze's cone calculus, see \cite{Schulze91, Schulze94}. For an analysis of stratified spaces in the abstract setting of Dirichlet spaces,
we also refer to \cite{Akuta}.


\subsection{Markovian property of  $\Delta_g$}\label{Section 3.2}

We write $\mathcal{T}\M$ for the $C^\infty(\M)$-module of all smooth sections of $T\M$, 
and
denote by 
$$
|\cdot|_g: \mathcal{T}\M \to C^\infty(\M), \quad a\mapsto \sqrt{(a|a)_g}
$$
the (vector bundle) norm induced by the Riemannian metric $g=(\cdot|\cdot)_g$. 
In addition, we define $\langle \cdot, \cdot \rangle_{\mathcal{T}}: \mathcal{T}\M\times \mathcal{T}\M\to \R$ by
$$
\langle u, v \rangle_{\mathcal{T}} := \int_\M (u | v)_g \, d\mu_g,
$$ 
where $d\mu_g$ is the volume element induced by $g$, and by $\langle \cdot ,\cdot \rangle$ the inner product of $L_2(\M)$.

Let us define a quadratic form associated with $-\Delta_g$ by
$$
\mathfrak{a}(u,v)=\langle \nabla u,  \nabla v \rangle_{\mathcal{T}},
$$
on $\mathscr{D}:=C_c^\infty(\M)$.
Here $\nabla u$ is the gradient vector.
We have
$$
\mathfrak{a}(u,v)= \langle -\Delta_g u, v \rangle, \quad u,v\in \mathscr{D}.
$$
Moreover, $-\Delta_g$ is symmetric on $\mathscr{D}$. By \cite[Theorem~4.14]{Dav80} and \cite[Theorem~1.2.8]{Dav89}, $\mathfrak{a}$ is closable and its closure, still denoted by $\mathfrak{a}$, with domain
\begin{center}
$D(\mathfrak{a})=$ the completion of $\mathscr{D}$ with respect to $\|\cdot\|_{D(\mathfrak{a})}$ in $L_{1,loc}(\M)$,
\end{center}
where $\|u\|_{D(\mathfrak{a})}:= \|u\|_2 + \||\nabla u|_g\|_2 $, is associated with a self-adjoint operator, i.e., the Friedrichs extension of $-\Delta_g$. 
We denote the Friedrichs extension of $\Delta_g$ by $\FD$.  
Interested readers may refer to the Sections~\ref{Section 3.3} and \ref{Section 4.2} for more details of $D(\mathfrak{a})$ and $D(\FD)$.


Since $-\FD$ is non-negative and self-adjoint, from the spectral theory, $\FD$ generates a self-adjoint contraction $C_0$-semigroup  $\{e^{t\FD}\}_{t\geq 0}$ on $L_2(\M)$.

In the sequel, we will establish the Markovian property of the semigroup $\{e^{t\FD}\}_{t\geq 0}$. 
To this end, let us first introduce several concepts on Dirichlet forms and Markovian semigroups.

\begin{definition}
A symmetric quadratic form $\cE$ defined in $L_2(\M)$ with domain $D(\cE)$ is called  Markovian if for each $\varepsilon>0$ there exists $\phi_\varepsilon:\R\to \R$ such that $-\varepsilon\leq \phi_\varepsilon(t) \leq 1+\varepsilon$ for all $t\in\R$ and $\phi_\varepsilon(t)=t$ for $t\in [0,1]$ and
$$
0\leq \phi_\varepsilon(t)-\phi_\varepsilon(s) \leq t-s \quad \text{whenever } t>s,
$$
and 
$$
u\in D(\cE)\Longrightarrow \phi_\varepsilon(u)\in D(\cE),\quad \cE(\phi_\varepsilon(u),\phi_\varepsilon(u))\leq \cE(u,u).
$$
A closed symmetric Markovian quadratic form is called a  Dirichlet form.
\end{definition}

Let $X_\R$ be a real Banach lattice with an order $\leq$. See \cite[Chapter~C-I]{ArenGrohNage86}. The complexification of $X_\R$ is a complex Banach lattice defined by
\begin{equation}
\label{S4.1: Banach lattice}
X:=X_\R \oplus i X_\R. 
\end{equation}
The positive cone of $X_\R$ is defined by
$$X_\R^+:=\{x\in X_\R:\, 0\leq x\}. $$
\begin{definition}
Let $\vartheta\in\R$, and $X$ be a complex Banach lattice defined as in \eqref{S4.1: Banach lattice}.
A semigroup $T(t)$ is called real if 
$$T(t) X_\R \subset X_\R ,\quad t\geq 0.$$ 
We say that $T(t)$ is positive if 
$$T(t) X_\R^+ \subset X_\R^+  ,\quad t\geq 0.$$ 
\end{definition}

\begin{definition}\label{Def: Mark semigroup}
A strongly continuous semigroup $T(t)$ on $L_2(\M)$ is called  Markovian if it is both positive and $L_\infty$-contraction, i.e.  
$$
\|T(t) u\|_\infty \leq \|u\|_\infty,\quad t\geq 0, \quad u\in  L_\infty(\M) \cap L_2(\M).
$$
\end{definition}

\begin{prop}\label{S4: Delta-positivity}
The semigroup $\{e^{t\FD}\}_{t\geq 0}$ is Markovian.
\end{prop}
\begin{proof}
By \cite[Example 1.2.1]{FukOshTak}, $\mathfrak{a}$ with domain $D(\mathfrak{a})$ is Markovian and thus is a Dirichlet form. 
It follows from \cite[Theorem~1.4.1]{FukOshTak} that 
\begin{equation}\label{S3.2: markov semigroup property}
0\leq u \leq 1 \Longrightarrow 0\leq e^{t\FD} u\leq 1.
\end{equation} 
Note that although \cite[Theorem~1.4.1]{FukOshTak} is proved  for  the  quadratic form $\cE(u,v)=\langle (-\FD)^{1/2} u , (-\FD)^{1/2} v \rangle$, one  can follow that proof and check that it also works for $\mathfrak{a}$.

It only remains to show that $\{e^{t\FD}\}_{t\geq 0}$ is positivity-preserving. For any $0\leq u \in L_2(\M)$, take a sequence $(u_n)$ in $L_\infty(\M)$ converging to $u$ in  $L_2(\M)$. Without loss of generality, we can take $u_n\geq 0$. By the above discussion, $e^{t\FD} u_n\geq 0$. If we have $\|(e^{t\FD} u)^- \|_2 \geq C>0$, then
$$
\|e^{t\FD} (u_n-u)\|_2 \geq \|(e^{t\FD} u)^- \|_2 \geq C>0.
$$
A contradiction. Therefore, $\{e^{t\FD}\}_{t\geq 0}$ is  Markovian.
\end{proof}

\begin{remark}
Although \cite[Example~1.2.1, Theorem~1.4.1]{FukOshTak} are stated for real-valued function spaces, we can overcome this restriction by the following procedure. First consider   $\Delta_g : C_c^\infty(\M;\R) \to L_2(\M;\R)$  and prove that its associated quadratic form $\mathfrak{a}_\R$ is Dirichlet. Then denoting the corresponding self-adjoint extension  related to $\mathfrak{a}_\R$ by $\underline{\Delta}_{F,\R}$, it is an easy job to check that
$$
\underline{\Delta}_{F,\R}=\FD|_{D(\underline{\Delta}_{F,\R})} \quad \text{and} \quad 
e^{t \underline{\Delta}_{F,\R}}= e^{t \FD}|_{L_2(\M;\R)},
$$
where $D(\underline{\Delta}_{F,\R})$ is the restriction of $D(\FD)$ on real-valued functions.
Now it follows from the above proof that \eqref{S3.2: markov semigroup property} still holds true.
\end{remark}



\begin{prop}
\label{S4: contraction semigroup-Lp}
$\Delta_g$ generates a contraction $C_0$-semigroup on $L_p(\M)$  for  $1\leq p< \infty$, 
and this semigroup is analytic when $1<p<\infty$. Moreover, for any $\omega>0$, $0\in \rho(\Delta_g-\omega)$ and
\begin{equation}
\label{S4: Lap semigroup est}
\|e^{t(\Delta_g-\omega)} \|_{\cL(L_p(\M))} \leq e^{-\omega t}.
\end{equation}
\end{prop}
\begin{proof}
The generation of semigroup part follows from a standard argument, cf. 
\cite[Section~1.4]{Dav89}.
The assertion that $0\in \rho(\Delta_g-\omega)$ is a direct consequence of the Hille-Yosida theorem.
\eqref{S4: Lap semigroup est} follows from the fact that $\{e^{t\Delta_g}\}_{t\geq 0}$ is a contraction.
\end{proof}

Throughout the rest of this paper, we will denote by  $\underline{\Delta}_{F,p}$ the infinitesimal generator of the semigroup  obtained via the extension of $\{e^{t\FD} \}_{t\geq 0}$ onto  $L_p(\M)$ in Proposition~\ref{S4: contraction semigroup-Lp} and its domain by $D(\underline{\Delta}_{F,p})$. In particular, $\underline{\Delta}_{F,2}=\FD$.
Note that  $D(\underline{\Delta}_{F,p})$ is always  a subspace of $D(\underline{\Delta}_{\max})$, cf. \eqref{S2.2: max domain}.


\subsection{Uniqueness of Markovian extension}\label{Section 3.3}

In this subsection, we will show that $\FD$ is indeed the unique Markovian extension of the unbounded operator $\Delta_g: C_c^\infty(\M)\to L_2(\M)$.

By \cite[Theorem~5.37]{Weid76}, the domain of $\FD$ can be expressed as
$$
D(\FD)= D(\mathfrak{a}) \cap D(\underline{\Delta}_\max),
$$
where $D(\mathfrak{a})$ is the domain of the quadratic form $\mathfrak{a}$ defined in Section~\ref{Section 3.2}.

A closed extension $A$ of the unbounded operator  $\Delta_g: C_c^\infty(\M) \to L_2(\M)$  is called a {\em Markovian extension} of $\Delta_g$ if $A$ generates a Markovian semigroup. From \cite[Theorem~1.3.1]{FukOshTak}, we learn that there is a one-one correspondence between Markovian semigroups and Dirichlet forms. Therefore, each  Markovian extension $A$ of  $\Delta_g$ is associated to a Dirichlet form $\cE_A$ defined on $L_2(\M)$.

Denote by $\cM_\Delta$ the set of all Markovian extensions of $\Delta_g$. We can define a partial order in $\cM_\Delta$   by
\begin{center}
$A \preceq B$ iff $D(\cE_A) \subseteq D(\cE_B)$ and $\cE_A(u,u) \geq \cE_B(u,u)$ for all $u\in D(\cE_A)$
\end{center}
for two Markovian extensions $A,B$ of $\Delta_g$. 
Here $\cE_A$ and $\cE_B$ are the associated Dirichlet forms to $A,B$, respectively, and $D(\cE_A)$ and $D(\cE_B)$ are their domains, respectively. 
From \cite[Lemma~3.3.1]{FukOshTak}, we know that $\FD$ is the minimal Markovian extension with respect to the order $\preceq$.

Recall that $\|v\|_{D(\mathfrak{a})}= \|v\|_2 + \||\nabla v|_g\|_2 $.  
\begin{lem}\label{A2: lem domain}
Define
$$
H^1(\M)=\{u\in L_2(\M):\, \| |\nabla u|_g\|_2<\infty \}.
$$
Then $H^1(\M)$ is the completion of $H^1(\M)\cap C^\infty(\M)$ with respect to $\|\cdot\|_{D(\mathfrak{a})}$  and thus is a Banach space.
\end{lem}
\begin{proof}
The proof is exactly the same to that of \cite[Lemma~3.3.3]{FukOshTak}.
\end{proof}
 
\begin{lem}\label{A2: D(a) char}
For any $ \gamma \neq  \frac{n-1}{2}  $,  
$$
\|u\|_{X^\gamma_p}:= \|  u \|_{\cH^{0,\gamma}_p } + \||\nabla   u |_g \|_{\cH^{0,\gamma}_p} 
$$
defines an equivalent norm on $\cH^{1,\gamma+1}_p(\M)$.
In particular, we have
$D(\mathfrak{a}) \doteq \cH^{1,1}_2(\M)$ for $n\neq 1$.
\end{lem}
\begin{proof}
It suffices to consider this problem inside $I\times \B$ and thus we only consider $u\in \cH^{1,\gamma+1}_p(\M)$ vanishing outside $I\times \B$. 
For this, in the rest of this proof, we are free to assume all functions vanish outside $I\times \B$, where $I=(0,1]$.  Thus, we can omit the multiplier $\psi$ in the definition of $\|\cdot\|_{X^\gamma_p}$.

By the weighted Hardy inequality cf. \cite{HLP34}, we have for $p>1$ and $a\neq 1$
\begin{equation}
\label{A2: WHI}
\int_0^\infty |f|^p x^{-a}\, dx \leq (\frac{p}{|a-1|})^p \int_0^\infty |f^\prime(x)|^p x^{p-a}\, dx.
\end{equation}
We will show that an analogue of \eqref{A2: WHI} holds true for any $  f \in C^\infty_c(\M) $. Indeed,  we have
\begin{align*}
\int_{I\times \B} x^{-a} | f |^p\, d\mu_g  = & \int_\B \int_{(0,\infty)  }  x^{n-a} |   f |^p \, dxdy \\
\leq & (\frac{p}{|a-1-n|})^p \int_\B \int_{(0,\infty)  } x^{p+n-a } | \partial_x   f (x,y)  ) |^p \, dxdy \\
= & (\frac{p}{|a-1-n|})^p \int_{I\times \B} x^{p-a } | \partial_x   f (x,y)  ) |^p \, d\mu_g.
\end{align*}
We infer that
\begin{equation}
\label{A2: WHI-2}
\int_{I\times \B} |f|^p x^{-a}\, d\mu_g \leq (\frac{p}{|a-1-n|})^p\int_{I\times \B} |\nabla f |_g^p x^{p-a}\, d\mu_g 
\end{equation}
for all $  f \in C^\infty_c(\M) $ and $a\neq n+1$. Here we have abused the notation a little and omitted the pull-back of $f$ into the local coordinates of $\B$.
We will use this convention in the sequel. 

Recall that $f\in \cH^{1,\gamma+1}_p(\M)$ iff  $f\in H^1_{p,loc}(\M)$   such that 
$$
\|f\|_{\cH^{1,\gamma+ 1}_p}^p = \sum_{k=0}^1 \int_{I\times\B} x^{p(k + \frac{n-1}{2} -\gamma -\frac{n+1}{p}) } |\nabla^k f|_g^p \, d\mu_g < \infty.
$$
Given any $ f \in C^\infty_c(\M)$, it is clear that
$$
\|f\|_{X^\gamma_p} \leq \|f\|_{\cH^{1,\gamma+1 }_p} .
$$
Choosing $a=p(  \gamma + \frac{n+1}{p} - \frac{n-1}{2})$ in \eqref{A2: WHI-2}, we can infer that
$$
\|f\|_{\cH^{1,\gamma+1}_p} \leq C(n) \|f\|_{X^\gamma_p}  
$$
for $\gamma \neq  \frac{n-1}{2}$.
It follows from Lemma~\ref{S2.1: lem-isometry} that $u\in \cH^{1,\gamma+1}_p(\M) $ iff $S_{\gamma+1} u \in H^1_p(\R_+\times \B)$. 
On the infinite cylinder $\R_+\times \B$, functions in $H^1_p(\R_+\times \B)$ can be approximated by $C^\infty_c$-functions. Therefore, the closure of $C_c^\infty(\M)$ with respect to $\|\cdot\|_{\cH^{1,\gamma+1}_p}$ is exactly $\cH^{1,\gamma+1}_p(\M)$. This proves the asserted statement.
\end{proof}

We introduce a new closed symmetric quadratic form
$$
\cE_+(u,v)= \langle \nabla u, \nabla v \rangle_{\mathcal{T}}
$$
with domain $H^1(\M)$, and denote the associated Markovian extension by $A_+$.
One can follow the proof of \cite[Theorem~3.3.1]{FukOshTak} and show that $A_+$ is the maximal element of  $\cM_\Delta$.

\begin{prop}\label{Prop: unique Markov}
$H^1(\M)=D(\mathfrak{a})$ and consequently, $\underline{\Delta}_F$ is the unique Markovian extension of $\Delta_g$.
\end{prop}
\begin{proof}
As in the proof of Lemma~\ref{A2: D(a) char}, we assume all functions vanish outside $I\times \B$. 

{\bf Case 1:} $n=1$ 

This result is proved in \cite[Proposition~3.10]{BosPra16} for $\B=S^1$; and it is  known that any one-dimensional connected closed manifold is homeomorphic to $S^1$.

{\bf Case 2:} $n>1$
 
It is not hard to show that $H^1(\M)$ is the completion of $H^1(\M)\cap L_\infty(\M)$. Indeed, given any $k\in \bN$ and $u\in H^1(\M)$, we can find $\lambda_k>0$ such that
$$
\|u|_{\{|u| > \lambda_k\}} \|_{D(\mathfrak{a})} <1/k.
$$
Define 
\begin{align*}
u_k (\mathsf{p}) =
\begin{cases}
\lambda_k \quad & \text{when } u(\mathsf{p}) >\lambda_k \\
u(\mathsf{p}) & \text{when } |u(\mathsf{p}) |\leq \lambda_k\\
-\lambda_k \quad & \text{when } u(\mathsf{p}) <-\lambda_k .
\end{cases}
\end{align*}
Then $u_k \to u$ in $H^1(\M)$.

Next we will show that $H^1(\M)\cap L_\infty(\M)\subseteq \cH^{1,1}_2(\M)$. First, any $u\in H^1(\M) $ satisfies 
$$
\int_{I\times \B} |\nabla u|^2 \, d\mu_g <\infty.
$$
On the other hand, if we assume further $u\in   L_\infty(\M)$, it holds that
\begin{align*}
\int_{I\times \B} x^{-2} |u|^2 \, d\mu_g  = \int_\B \int_I x^{n-2} |u|^2 \, dxdy \leq \|u\|^2_\infty {\rm vol}(\B).
\end{align*}
Combining with Lemma~\ref{A2: D(a) char}, this shows that $D(\mathfrak{a})$ is dense in $H^1(\M)$ and thus $D(\mathfrak{a})=H^1(\M)$.
\end{proof}


\subsection{Functional inequalities in Mellin-Sobolev spaces}\label{Section 3.4}


\begin{prop}\label{Prop: Poincare}
For any $ \gamma\neq \frac{n-1}{2}$, there exists a constant $C>0$ such that 
$$
\|u-(u)_{\theta,\gamma}\|_{\cH^{0,\gamma}_p}\leq C \||\nabla u|_g\|_{\cH^{0,\gamma}_p} ,\quad u\in \cH^{1,\gamma+1}_p(\M),
$$
where $\displaystyle (u)_{\theta,\gamma}= \frac{\int_\M \tilde{x}^{\theta-\gamma} u\, d\mu_g}{\int_\M \tilde{x}^{\theta-\gamma}  \, d\mu_g} $  with $\theta>   \frac{n+1}{2} - \frac{n+p}{p}$ and 
\begin{align*}
\tilde{x}(\mathsf{p})=
\begin{cases}
x ,\quad &\mathsf{p}=(x,y) \in I\times \B \\
1 , & \text{elsewhere}.
\end{cases}
\end{align*} 
In particular, we have
$$
\|u-\bar{u}\|_2\leq C \||\nabla u|_g\|_2 ,\quad u\in \cH^{1,1}_2(\M),
$$
where $\bar{u}$ is the mean of $u$ on $\M$.
\end{prop}
\begin{proof}
Assume, to the contrary, that the theorem fails. Then for any $k\in \bN$, there exists $u_k\in \cH^{1,\gamma+1}_p(\M)$ such that
$$
\|u_k-(u_k)_{\theta,\gamma}\|_{\cH^{0,\gamma}_p}\ >k \||\nabla u_k|_g\|_{\cH^{0,\gamma}_p}  .
$$
Let 
$$
v_k:= \frac{u_k-(u_k)_{\theta,\gamma}}{\|u_k-(u_k)_{\theta,\gamma}\|_{\cH^{0,\gamma}_p}}.
$$
Then it holds that $(v_k)_{\theta,\gamma}=0$, $\|v_k\|_{\cH^{0,\gamma}_p}=1$ and 
$$
\||\nabla v_k|_g\|_{\cH^{0,\gamma}_p} < 1/k.
$$
Together with Lemma~\ref{A2: D(a) char}, this implies that $\|v_k\|_{X^\gamma_p}$, and thus $\|v_k\|_{\cH^{1,\gamma+1}_p}$, are uniformly bounded.
It follows from   Proposition~\ref{S2.1: Rellich-thm}  that we can find some $v\in \cH^{1,\gamma+1}_p(\M)$ such that $v_k \to v$ in $\cH^{0,\gamma}_p(\M)$ and $v_k \rightharpoonup v$ in $\cH^{1,\gamma+1}_p(\M)$. 
The condition $\theta>   \frac{n+1}{2} - \frac{n+p}{p}$  guarantees that 
$$
\cH^{0,\gamma}_p(\M) \hookrightarrow \cH^{0,\gamma- \frac{n+1}{2}-\theta}_1(\M) 
$$
due to the fact  $- \frac{n+1}{2}-\theta<0$ and \eqref{Sobolev embedding}. 
Note that
$$
[u \mapsto \int_\M \tilde{x}^{\theta-\gamma} |u| \, d\mu_g]
$$
is an equivalent norm on $\cH^{0,\gamma- \frac{n+1}{2}-\theta}_1(\M)$.
We thus infer that $(v)_{\theta,\gamma}=0$. Moreover, since $\| | \nabla v|_g\|_{\cH^{0,\gamma}_p}=0$, we conclude that $v\equiv 0$. A contradiction.
\end{proof}

Given $p\in [1,n+1)$, let $p^*=\frac{pn+p}{n+1-p}$.
\begin{prop}
Assume that $ \gamma\neq \frac{n-1}{2}$ and $\theta>   \frac{n+1}{2} - \frac{n+p}{p}$ . Then the following inequalities hold.
\begin{itemize}
\item[(i)] {\em (Sobolev-Poincar\'e inequality):} When $p\in [1,n+1)$,
\begin{equation}
\label{S3.4: Sobolev-Poincare}
\|u-(u)_{\theta,\gamma}\|_{\cH^{0,\gamma}_{p^*}} \leq C  \||\nabla u|_g\|_{\cH^{0,\gamma}_p} ,\quad u\in \cH^{1,\gamma+1}_p(\M) .
\end{equation}
\item[(ii)] {\em (Nash inequality):} When $n\geq 2$,
\begin{equation}
\label{S3.4: Nash ineq}
\|u-(u)_{\theta,\gamma}\|_2^{1+\frac{2}{n+1}} \leq C_1 \||\nabla u|_g\|_{\cH^{0,\gamma}_2} \|u-(u)_{\theta,\gamma}\|_{\cH^{0,\gamma}_1}^{\frac{2}{n+1}} ,\quad u\in \cH^{1,\gamma+1}_p(\M).
\end{equation}
\item[(iii)] {\em (Super Poincar\'e inequality):} When $n\geq 2$
\begin{equation}
\label{S3.4: Super Poincare}
\|u-(u)_{\theta,\gamma}\|_2^2 \leq r \|u-(u)_{\theta,\gamma}\|_1^2 + \beta(r) \| |\nabla u|_g \|_2^2,  \quad u\in \cH^{1,\gamma+1}_p(\M),
\end{equation}
where $\beta: (0,\infty)\to (0,\infty)$ is a decreasing function.
\end{itemize}
\end{prop}
\begin{proof}
\eqref{S3.4: Sobolev-Poincare} is a direct consequence of \eqref{Sobolev embedding}, Lemma~\ref{A2: D(a) char} and Proposition~\ref{Prop: Poincare}.
By the H\"older inequality, we have
$$
\|u-(u)_{\theta,\gamma}\|_{\cH^{0,\gamma}_2} \leq \|u-(u)_{\theta,\gamma}\|_{\cH^{0,\gamma}_{2^*}}^\vartheta \|u-(u)_{\theta,\gamma}\|_{\cH^{0,\gamma}_1}^{1-\vartheta},
$$
where $\vartheta=\frac{n+1}{n+3}$ and $2^*=\frac{2n+2}{n-1}$. Taking $p=2$ in \eqref{S3.4: Sobolev-Poincare} gives \eqref{S3.4: Nash ineq}.
Then \eqref{S3.4: Super Poincare} follows from the Young's inequality and \eqref{S3.4: Nash ineq}.
\end{proof}



\section{Fractional Powers of the Laplace-Beltrami Operator}\label{Section 4}


\subsection{Construction of the Fractional Laplace-Beltrami Operators}\label{Section 4.1}


Proposition~\ref{S4: contraction semigroup-Lp} implies that the operator $\underline{\Delta}_{F,p}-\omega$ has a bounded inverse in $L_p(\M)$ when $\omega>0$.  
In this case, 
we can define $(\omega-\underline{\Delta}_{F,p})^\sigma$ by
\begin{align}
\label{S4: fractional Delta def}
(\omega-\underline{\Delta}_{F,p})^\sigma :&=\frac{\sin(\pi\sigma)}{\pi} \int_0^\infty s^{\sigma-1} (\omega - \underline{\Delta}_{F,p})(s+\omega-\underline{\Delta}_{F,p})^{-1} \, ds\\
\label{S4: fractional Delta def 2}
  &= \frac{1}{\Gamma(1-\sigma)}\int_0^\infty s^{-\sigma}(\omega-\underline{\Delta}_{F,p}) e^{s(\underline{\Delta}_{F,p}-\omega)}\, ds
\end{align}
in $D(-\underline{\Delta}_{F,p})$,
cf. \cite[Formula~(2.33)]{Tana}. 
\eqref{S4: fractional Delta def 2} is a direct consequence of \eqref{S4: fractional Delta def} and \cite[Formula~(2.1.10)]{Lunar95}, i.e.
$$
(s+\omega-\underline{\Delta}_{F,p})^{-1}=\int_0^\infty e^{-st} e^{t (\underline{\Delta}_{F,p}-\omega)}\, dt.
$$
In particular, by \cite[Formula~(2.34)]{Tana}  the following resolvent formula  holds.
\begin{equation}
\label{S4: Resovlent fractional lap}
(\lambda + (\omega-\underline{\Delta}_{F,p})^\sigma)^{-1} =\frac{\sin(\pi\sigma)}{\pi} \int_0^\infty \frac{s^\sigma (s+\omega-\underline{\Delta}_{F,p})^{-1}  }{(s^\sigma + \lambda e^{i\pi\sigma})(s^\sigma + \lambda e^{-i\pi\sigma}) }    \, ds, \quad \lambda>0.
\end{equation}
When $\omega=0$, we define the pseudo-resolvent $I(\lambda)$ for $\lambda\in (0,\infty)$ by
$$
I(\lambda)= \frac{\sin(\pi\sigma)}{\pi} \int_0^\infty \frac{s^\sigma (s-\underline{\Delta}_{F,p})^{-1}  }{(s^\sigma + \lambda e^{i\pi\sigma})(s^\sigma + \lambda e^{-i\pi\sigma}) }    \, ds.
$$
Then $(-\underline{\Delta}_{F,p})^\sigma$ is defined as the unique closed densely defined operator such that 
$$
(\lambda + (-\underline{\Delta}_{F,p})^\sigma)^{-1}=I(\lambda),\quad \lambda\in (0,\infty).
$$
Moreover, \eqref{S4: Resovlent fractional lap} still holds for $\omega=0$, see \cite[Formulas~(2.40) and (2.44)]{Tana}.

The domain of $(\omega-\underline{\Delta}_{F,p})^\sigma$ is independent of $\omega\geq 0$, cf. \cite[Lemma~2.3.5]{Tana}. For this, we will simply abbreviate it to $D(( -\underline{\Delta}_{F,p})^\sigma)$.  
Further, \cite[Lemma~2.3.5]{Tana} states that
\begin{equation}\label{S4: control diff}
\| (\omega-\underline{\Delta}_{F,p})^\sigma u - (-\underline{\Delta}_{F,p})^\sigma u\|_p \leq c \omega^\sigma \|u\|_p,\quad u \in D(( -\underline{\Delta}_{F,p})^\sigma).
\end{equation}
Next, we will show that Formula~\eqref{S4: fractional Delta def} actually holds for $\omega=0$. Indeed first note that Proposition~\ref{S4: contraction semigroup-Lp} and \cite[Formula~(2.1.10)]{Lunar95}  imply
\begin{align*}
\| ((\lambda+\omega-\underline{\Delta}_{F,p})^{-1} u\|_p &= \|\int_0^\infty e^{-\lambda t} e^{t ( \underline{\Delta}_{F,p}-\omega)} u\, dt\|_p \\
&\leq\int_0^\infty e^{-\lambda t} \|u\|_p \,dt = \frac{1}{\lambda} \|u\|_p ,
\end{align*}
which gives
\begin{equation}
\label{Lp resolvent est}
\| \lambda(\lambda+\omega-\underline{\Delta}_{F,p})^{-1}\|_{\cL(L_p(\M))}\leq 1, \quad \lambda>0,\, \omega\geq 0.
\end{equation}
This implies that the operator  
$$
B u :=-\frac{\sin(\pi\sigma)}{\pi} \int_0^\infty s^{\sigma-1}  \underline{\Delta}_{F,p}(s -\underline{\Delta}_{F,p})^{-1} u \, ds
$$
is well-defined for all $u\in D(-\underline{\Delta}_{F,p})$ because we can write $Bu$ as
\begin{align*}
Bu= & -\frac{\sin(\pi\sigma)}{\pi} \int_1^\infty s^{\sigma-1}  \underline{\Delta}_{F,p}(s -\underline{\Delta}_{F,p})^{-1} u \, ds \\
& +  \frac{\sin(\pi\sigma)}{\pi} \int_0^1 s^{\sigma-1}  u   \, ds -\frac{\sin(\pi\sigma)}{\pi} \int_0^1 s^\sigma  (s -\underline{\Delta}_{F,p})^{-1} u \, ds
\end{align*}
and all three terms converge absolutely in $L_p(\M)$. \eqref{S4: fractional Delta def} yields
\begin{align*}
& \| (\omega-\underline{\Delta}_{F,p})^\sigma u - B u\|_p \\
\leq & \frac{\sin(\pi\sigma)}{\pi} \omega \| \int_0^\infty s^{\sigma-1} (s+\omega-\underline{\Delta}_{F,p})^{-1} u \, ds \|_p \\
& + \frac{\sin(\pi\sigma)}{\pi}  \| \int_0^\infty s^{\sigma-1} \underline{\Delta}_{F,p} [(s-\underline{\Delta}_{F,p})^{-1} - (s+\omega -\underline{\Delta}_{F,p})^{-1}]  u\, ds \|_p .
\end{align*}
The first term on the right hand side can be estimated as follows. 
\begin{align*}
& \omega \| \int_0^\infty s^{\sigma-1} (s+\omega-\underline{\Delta}_{F,p})^{-1} u \, ds \|_p \\
\leq &  \int_\omega^\infty s^{\sigma-1} \omega \| (s+\omega-\underline{\Delta}_{F,p})^{-1} u\|_p  \, ds  + \int_0^\omega s^{\sigma-1} (s+\omega)\| (s+\omega-\underline{\Delta}_{F,p})^{-1} u\|_p  \, ds \\
& + \int_0^\omega s^\sigma  \| (s+\omega-\underline{\Delta}_{F,p})^{-1} u\|_p  \, ds \\
\leq & c \omega^\sigma \|u\|_p.
\end{align*}
Similarly, 
\begin{align*}
 & \| \int_0^\infty s^{\sigma-1} \underline{\Delta}_{F,p} [(s-\underline{\Delta}_{F,p})^{-1} - (s+\omega -\underline{\Delta}_{F,p})^{-1}]  u\, ds \|_p  \\
 \leq & \int_0^\infty  s^\sigma \omega \| (s-\underline{\Delta}_{F,p})^{-1}   (s+\omega -\underline{\Delta}_{F,p})^{-1} u\|_p \, ds \\
 & +  \int_0^\infty   s^{\sigma-1} \omega    \| (s+\omega -\underline{\Delta}_{F,p})^{-1} u\|_p \, ds \\
 \leq & c \omega^\sigma \|u\|_p.
\end{align*}
To sum up, we have proved
$$
(\omega-\underline{\Delta}_{F,p})^\sigma u \to Bu \quad \text{in } L_p(\M)$$
as $\omega \to 0^+$. In view of \eqref{S4: control diff}, we thus show that $B= ( -\underline{\Delta}_{F,p})^\sigma$ in 
$D(-\underline{\Delta}_{F,p})$. This coincides with Balakrishnan's  formula for the fractional power of a dissipative operator, cf. \cite{Bal60}.


\subsection{Properties of the Fractional Laplace-Beltrami Operator}\label{Section 4.2}

\begin{lem}
\label{S4: Positivity}
For all $\lambda>0$, $\omega \geq 0$, $\sigma\in (0,1)$, and $u\in L_1(\M)$, we have
$$ 
\sup [\id +\lambda (\omega- \underline{\Delta}_{F,1})^\sigma]^{-1}u \leq \max \{0, \sup u\}.
$$
\end{lem}
\begin{proof}
As in the proof of \eqref{Lp resolvent est}, it follows from Proposition~\ref{S4: Delta-positivity} and \cite[Formula~(2.1.10)]{Lunar95}  that for all $u\in L_\infty(\M)$
\begin{align*}
\| (\lambda - \underline{\Delta}_{F,1})^{-1} u\|_\infty  &= \|\int_0^\infty e^{-\lambda t} e^{t \underline{\Delta}_{F,1}} u\, dt\|_\infty \\
&\leq\int_0^\infty e^{-\lambda t} \|u\|_\infty \,dt = \frac{1}{\lambda} \|u\|_\infty ,
\end{align*}
because $ L_\infty(\M)\hookrightarrow  L_1(\M)$ and $e^{t\underline{\Delta}_{F,1}}$ coincides with $e^{t\underline{\Delta}_{F}}$ on $L_\infty(\M)$.
For any $\lambda>0$ and $u\in L_\infty(\M)$, \eqref{S4: Resovlent fractional lap} yields
\begin{align*}
&\| [\lambda + (\omega-\underline{\Delta}_{F,1})^{\sigma}]^{-1} u \|_\infty \\
 \leq  &\frac{\sin(\pi\sigma)}{\pi} \int_0^\infty \frac{s^{\sigma-1}}{(s^\sigma + \lambda e^{i\pi\sigma})(s^\sigma + \lambda e^{-i\pi\sigma}) } \| s (s+\omega-\underline{\Delta}_{F,1})^{-1} u \|_\infty \, ds\\
\leq & \|u\|_\infty \frac{\sin(\pi\sigma)}{\pi} \int_0^\infty \frac{s^{\sigma-1}}{(s^\sigma + \lambda e^{i\pi\sigma})(s^\sigma + \lambda e^{-i\pi\sigma}) }  \, ds
= \frac{1}{\lambda}  \|u\|_\infty .
\end{align*}
Therefore, for all $u\in L_1(\M)$, we have
$$
\| [\id + \lambda(\omega-\underline{\Delta}_{F,1})^{\sigma}]^{-1} u \|_\infty \leq \|u\|_\infty.
$$
Next, we will use Proposition~\ref{S4: Delta-positivity} to establish the positivity of $\{e^{t\underline{\Delta}_{F,1}}\}_{t\geq 0}$  on $L_1(\M)$.
Pick $u\in L_1(\M,\R_+)$ and a sequence $(u_k)_k$ in $L_2(\M)$ converging to $u$ in $L_1(\M)$. Without loss of generality, we may assume $u_k\in L_2(\M,\R_+)$.
If 
$$
\| (e^{t\underline{\Delta}_{F,1}}u)^-\|_1 \geq C>0,
$$
then by the positivity of $\{e^{t\FD}\}_{t\geq 0}$ on $L_2(\M)$ and Proposition~\ref{S4: contraction semigroup-Lp}, we have
$$
\| e^{t\underline{\Delta}_{F,1}}(u-u_k)\|_1 \geq \| (e^{t\underline{\Delta}_{F,1}}u)^-\|_1 \geq C>0.
$$ 
A contradiction. This implies that $\{e^{t\underline{\Delta}_{F,1}}\}_{t\geq 0}$ is positive on $L_1(\M)$.
For any $u\in L_1(\M,\R_+)$, we have
$$
(\lambda-\underline{\Delta}_{F,1})^{-1} u =\int_0^\infty e^{-\lambda t} e^{t\underline{\Delta}_{F,1} }u\, dt \geq 0,\quad \lambda>0.
$$
This gives the positivity of $(\id -  \lambda\underline{\Delta}_{F,1})^{-1}$ on $L_1(\M)$, i.e., for all $\lambda>0$
\begin{equation}
\label{S4: Positivity of resl Delta L1}
(\id-  \lambda\underline{\Delta}_{F,1})^{-1} (L_1(\M,\R_+))\subset L_1(\M,\R_+).
\end{equation}
The positivity of  $[\id +  \lambda(\omega-\underline{\Delta}_{F,1})^\sigma]^{-1}$ on $L_1(\M)$  follows  from \eqref{S4: Resovlent fractional lap} and \eqref{S4: Positivity of resl Delta L1}.

Given any $u\in L_1(\M,\R)$, we decompose it into $u=u^+ - u^-$. 
It holds that
\begin{align*}
&\sup[\id + \lambda(\omega-\underline{\Delta}_{F,1})^\sigma]^{-1} u^+\\
 \leq& \|[\id + \lambda(\omega- \underline{\Delta}_{F,1})^\sigma]^{-1} u^+\|_\infty \leq \|u^+\|_\infty= \sup u^+.
\end{align*}
On the other hand, since $-u^-\leq 0$, it follows from the positivity of $[\id + \lambda(\omega-\underline{\Delta}_{F,1})^\sigma]^{-1} $ on $L_1(\M)$ that
$$
\sup [\id + \lambda(\omega-\underline{\Delta}_{F,1})^\sigma]^{-1} (-u^-)\leq 0.
$$
Combining  these two estimates, the asserted claim then follows.
\end{proof}

\begin{prop}
\label{S4: Lap-frac semigroup}
For $1\leq p< \infty$ and $\omega\geq 0$, $-(\omega-\underline{\Delta}_{F,p})^\sigma$ generates a contraction $C_0$-semigroup on $L_p(\M)$. In particular, when $1<p<\infty$, this semigroup is analytic;
when $p=2$,  this semigroup is Markovian.
Furthermore, $0\in \rho((\omega-\underline{\Delta}_{F,p})^\sigma)$ whenever $\omega>0$.
\end{prop}
\begin{proof}
From Proposition~\ref{S4: contraction semigroup-Lp}, \cite[Theorems~1.4.1 and 1.4.2]{Dav89} and the standard semigroup theory, for all $1\leq p<\infty$ and $\omega>0$, we can learn the following facts:
\begin{itemize}
\item $\omega-\underline{\Delta}_{F,p}$ is densely defined and closed.
\item $\R_+\setminus\{0\}  \subset \rho(\underline{\Delta}_{F,p}-\omega)$, and $\| \lambda(\lambda+\omega-\underline{\Delta}_{F,p})^{-1}\|_{\cL(L_p(\M))}\leq 1$  for any $\lambda>0$. 
Particularly, $0\in \rho(\underline{\Delta}_{F,p}-\omega)$ when $\omega>0$.
\item We can find some $\theta_p\in (0,\pi]$ and $M>0$ such that $\Sigma_{\theta_p}\subset \rho(\underline{\Delta}_{F,p}-\omega)$ and 
$$
\| \lambda(\lambda+\omega-\underline{\Delta}_{F,p})^{-1}\|_p \leq M , \quad \lambda\in \Sigma_{\theta_p}.
$$
In particular, $\theta_p>\pi/2$ when $1<p<\infty$.
\end{itemize}
The resolvent estimate $\| \lambda(\lambda+\omega-\underline{\Delta}_{F,p})^{-1}\|_p \leq 1$ for $\lambda>0$ is proved in \eqref{Lp resolvent est}.
In view of \cite[Formula~(2.26), Definition~2.3.1, Propositions~2.3.1, Theorem~2.3.1]{Tana}, we infer that for any $\sigma\in (0,1)$ the following properties are satisfied by $(\omega-\underline{\Delta}_{F,p})^\sigma$.
\begin{itemize}
\item $(\omega-\underline{\Delta}_{F,p})^\sigma$ is densely defined and closed.
\item $0\in \rho((\omega-\underline{\Delta}_{F,p})^\sigma)$ whenever $\omega>0$.
\item $\Sigma_{(1-\sigma)\pi+\sigma\theta_p}\subset \rho(-(\omega-\underline{\Delta}_{F,p})^\sigma)$, and for any $\lambda>0$
$$
\| \lambda(\lambda+(\omega-\underline{\Delta}_{F,p})^\sigma)^{-1}\|_p \leq 1.
$$
\item For every $\varepsilon>0$, there exists $M_\varepsilon>0$ such that for all $\lambda\in \Sigma_{(1-\sigma)\pi+\sigma\theta_p-\varepsilon}$
$$
\| \lambda(\lambda+(\omega-\underline{\Delta}_{F,p})^\sigma)^{-1}\|_p \leq M_\varepsilon.
$$
\end{itemize}
Now it follows from the Lumer-Philips theorem and \cite[Proposition~4.4]{EngNag} that $-(\omega-\underline{\Delta}_{F,p})^\sigma$ generates a contraction $C_0$-semigroup on $L_p(\M)$ and this semigroup is analytic when $1<p<\infty$.

In Lemma~\ref{S4: Positivity}, we have shown that $[\id + \lambda(\omega-\underline{\Delta}_{F,1})^\sigma]^{-1}$ is 
an  $L_\infty-$contraction and positive. The Markovian property of the semigroup $\{e^{-t (\omega-\underline{\Delta}_F})^\sigma \}_{t\geq 0}$ now follows by   \cite[Theorem~1.8.3]{Pazy}.
\end{proof}

The invertibility of $(\omega-\underline{\Delta}_{F,1})^\sigma$ for $\omega>0$ immediately implies the following lemma.
\begin{cor}
\label{S4: inverse}
For every $\omega>0$, there exists some $C>0$ such that 
$$
C\|u\|_1 \leq \|(\omega-\underline{\Delta}_{F,1})^\sigma u\|_1, \quad u\in D(( -\underline{\Delta}_{F,1})^\sigma).
$$
\end{cor}

Replacing $(\omega-\underline{\Delta}_{F,1})^\sigma$ by $\lambda(\omega- \underline{\Delta}_{F,1})^\sigma$ for $\lambda>0$, the new operator still satisfies Lemma~\ref{S4: Positivity} and Corollary~\ref{S4: inverse}. 

Let $\Phi(x)=|x|^{m-1}x$ and $\beta=\Phi^{-1}$. Then they are maximal monotone graphs in $\R^2$ containing $(0,0)$.
By Proposition~\ref{S4: Lap-frac semigroup}, Lemma~\ref{S4: Positivity} and Corollary~\ref{S4: inverse}, we can apply \cite[Theorem~1]{BreStr73} and prove the following proposition.
\begin{prop}
\label{S4: semilinear thm}
For any $f\in L_1(\M)$ and all $\lambda>0$, there exists a unique solution $u\in D((-\underline{\Delta}_{F,1})^\sigma)$ to 
$$ \lambda(\omega-\Delta_g )^\sigma u +\beta(u) =f .$$
Moreover, for any $f_1,f_2\in L_1(\M)$, the corresponding solutions $u_1,u_2$ satisfy  
$$
\|\beta(u_1)  - \beta(u_2)  \|_1 \leq \|f_1-f_2\|_1.
$$
\end{prop}


In next two propositions, we will give some characterizations of the domain of $( -\underline{\Delta}_{F,p})^\sigma$.
\begin{prop}
\label{S4: char domain of Delta}
It holds that $D(\FD)\hookrightarrow \cH^{2,1+\delta}_2(\M)+\bC_\psi$ for some $\delta>0$ sufficiently small. 
The sum with $\bC_\psi$ is present only when $n=1$, in which case it is a direct sum.
As a result, for any $\varepsilon>0$,
\begin{equation}
\label{S9: embed 1}
D((-\FD)^\sigma)\hookrightarrow  \cH^{2\sigma-\varepsilon,\sigma+\delta\sigma-\varepsilon}_2(\M)+\bC_\psi.
\end{equation}
In particular, for any $1\leq p\leq 2 $ 
\begin{equation}
\label{S9: embed 2}
D((-\FD)^\sigma)\xhookrightarrow{c} L_p(\M).
\end{equation}
\end{prop}
\begin{proof}
According to \cite[Corollary 5.4]{SchSei05}, 
\begin{eqnarray}\label{Friedrichs}
D(\underline{\Delta}_{F}) = \begin{cases}
D(\underline{\Delta}_{F,\min}) \oplus \bigoplus_{q^\pm_j\in I_{0}, q^\pm_j<0} \cE_{q^\pm_j}\oplus \bC_\psi,&n=1\\
D(\underline{\Delta}_{F,\min}) \oplus \bigoplus_{q^\pm_j\in I_{0}, q^\pm_j<\frac{n-1}{2}} \cE_{q^\pm_j},&n>1.
\end{cases}
\end{eqnarray}
Here $I_{0}$ is defined in \eqref{S3.1: I_gamma}, and $\underline{\Delta}_{F,\min}:=\underline{\Delta}^0_{0,\min}$ is the minimal closed extension of $\Delta_g$ on $L_2(\M)$.  \eqref{S3.1: min domain} implies
$$
\mathcal{H}^{2, 2}_{2}( \M)\hookrightarrow D(\underline{\Delta}_{F,\min})\hookrightarrow \bigcap_{\varepsilon>0}\mathcal{H}^{2, 2-\varepsilon}_{2}( \M).
$$
When $\displaystyle \frac{n-3}{2} < q^\pm_j <\frac{n-1}{2}$, it is an easy job to verify that $\cE_{q^\pm_j}\hookrightarrow \cH^{2,1+\delta}_2(\M) $ for some $\delta>0$ sufficiently small.
\eqref{S9: embed 1} then follows from \eqref{Friedrichs}.

By \cite[Proposition~2.2.15]{Lunar95} and \cite[Lemma~2.3.5]{Tana},  it holds that
$$
D((-\FD)^\sigma) \hookrightarrow (\cH^{0,0}_2(\M) , D(\FD))_{\alpha,2},\quad 0<\alpha<\sigma.
$$
Together with Lemma~\ref{S2.1: Sobolev-interpolation}, this  establishes $D((-\FD)^\sigma)\hookrightarrow L_p(\M)$ when $1\leq p\leq 2 $. 
To see \eqref{S9: embed 2} is compact,  when $n>1$, it is a direct consequence of Proposition~\ref{S2.1: Rellich-thm}. 
In the case $n=1$, in view of \eqref{S9: embed 1} and noting that $\bC_\psi \hookrightarrow \cH^{1,\gamma_p+s}_p(\M)$ for  $s>0$ sufficiently small, we conclude that compactness of the embedding \eqref{S9: embed 2} from  Proposition~\ref{S2.1: Rellich-thm}. 
\end{proof}

\begin{prop}
\label{S4: char domain of Delta-Lp}
For any $1\leq p<\infty$, it holds that 
$D(\underline{\Delta}_{F,p})  \hookrightarrow  \cH^{2,\delta+\gamma_p}_p(\M)$
for some $\delta>0$ sufficiently small. As a result, for any $\varepsilon>0$,
$$
D((-\underline{\Delta}_{F,p})^\sigma)\hookrightarrow  \cH^{2\sigma-\varepsilon,\ \delta \sigma+\gamma_p -\varepsilon}_p(\M).
$$
In particular,   for $1\leq q \leq p$
$$
D((-\underline{\Delta}_{F,p})^\sigma)\xhookrightarrow{c} L_q(\M).
$$
\end{prop}
\begin{proof}
From \eqref{S2.2: max domain}, one can conclude that for sufficiently small $\delta>0$ 
$$
D(\underline{\Delta}_{F,p})\hookrightarrow D(\underline{\Delta}_{\max}) \hookrightarrow  \cH^{2,\delta+\gamma_p}_p(\M).
$$
Then the assertion can be proved in an analogous way as in Proposition~\ref{S4: char domain of Delta}.
\end{proof}



\section{Weak Solutions to the Fractional Porous Medium Equation}\label{Section 5}

In this section, based on the work in Sections~\ref{Section 2}-\ref{Section 4}, we will apply the nonlinear semigroup theory and establish the existence of a unique global weak solution to \eqref{S1: FPME} with $\sigma\in (0,1)$ in the case $m\in (0,1)\cup (1,\infty)$. 
Here $m$ is from \eqref{S1: FPME}. 
Note that, when $m>1$, $\Phi(u)\in D((-\underline{\Delta}_{F,1})^\sigma)\hookrightarrow L_1(\M)$ shows that $u\in L_m(\M)\hookrightarrow L_1(\M)$.
Then in Section~\ref{Section 7} we will show this solution is indeed strong.
At the end of Section~\ref{Section 7}, the case $m=1$ will be handled separately.

\begin{definition}\cite[Chapter II.3]{Bar73}
\begin{itemize}
\item[(i)] A nonlinear operator $\cA: D(\cA)\subseteq X\to X$ defined in a Banach space $X$ is called accretive if for all $\lambda>0$
\begin{equation}
\label{S4: contraction}
\|(\id +\lambda \cA) x_1 - (\id +\lambda \cA) x_2\|_X\geq \|x_1-x_2\|_X,\quad x_1,x_2\in D(\cA).
\end{equation}
\item[(ii)] A nonlinear operator $\cA$ defined in a Banach space $X$ is called $m$-accretive if $\cA$ is accretive and it satisfies the range condition
$$
Rng(\id +\lambda \cA)=X, \quad \lambda>0.
$$
\end{itemize}
\end{definition}

Let $X= L_1(\M)$ and $\cA(u):=(\omega-\underline{\Delta}_{F,1})^\sigma\Phi(u)$   with domain 
$$
D(\cA)=\{u\in L_1(\M): \Phi(u)\in D((-\underline{\Delta}_{F,1})^\sigma)\}.
$$ 
Note that, when $m\geq 1$, $u\in D(\cA)$ iff $\Phi(u)\in D((-\underline{\Delta}_{F,1})^\sigma)$.

Proposition~\ref{S4: semilinear thm} implies that, for any $\lambda>0$, $(\id+\lambda\cA)^{-1}$ is a bijection between $D(\cA)$ and $L_1(\M)$ when $m\geq 1$, and \eqref{S4: contraction} is fulfilled. 
In the case $m<1$, given any $u\in D(\cA)$, it holds that
\begin{equation}
\label{S6: f}
f:=\lambda\cA (u) + u \in L_1(\M).
\end{equation}
Now assuming $f\in  L_1(\M)$, let
$u$ be the unique solution of \eqref{S6: f} obtained in Proposition~\ref{S4: semilinear thm}. 
We have $u=f-\lambda\cA (u) \in L_1(\M)$ and $\Phi(u)\in D((-\underline{\Delta}_{F,1})^\sigma)$, and thus $u\in D(\cA)$.
Therefore, $\cA$ is $m$-accretive. 

\begin{lem}
$D(\cA)$ is dense in $L_1(\M)$.
\end{lem}
\begin{proof}
Note that $\beta\in BC^{1/m}(\R)$ when $m\geq 1$ and $\beta\in C^{1-}(\R)$ when $0<m<1$. 

Case 1: $m\geq 1$. For any $w\in L_m(\M)$, there exists a sequence $(u_k)_k \subset D((- \underline{\Delta}_{F,1})^\sigma)$ converging to $\Phi(w)$ in  $L_1(\M)$. It follows that
$$
\|\beta(u_n)- w\|_1 \leq C \| (u_n-\Phi(w))^{1/m}\|_1 \leq C\| u_n-\Phi(w)\|_1^{1/m}.
$$
This proves that the closure  of $D(\cA)$ contains $L_m(\M)$. Since $L_m(\M)$ is dense in $L_1(\M)$, we infer that $D(\cA)$ is dense in $L_1(\M)$.

Case 2: $0<m<1$. 
Given arbitrary $w\in L_\infty(\M)$, there exists a  sequence   $(u_k)_k \subset C^\infty_c(\M) \subset D((- \underline{\Delta}_{F,1})^\sigma)$  converging to $\Phi(w) $ in  $L_1(\M)$ satisfying that $\|\beta(u_k)\|_\infty\leq 2\|w\|_\infty$. So the Lipschitz continuity of $\beta$ implies 
\begin{align*}
\|\beta(u_k) - w\|_1 \leq C(w) \| u_k  - \Phi(w)\|_1 \to 0
\end{align*}
as $k\to \infty$. As above, this implies the density of $D(\cA)$ in $L_1(\M)$. 
\end{proof}


In the sequel, we will always use $\Delta_g$ to denote the closed extensions $\underline{\Delta}_{F,p}$ whenever the choice of $p$ is clear from  context.

With the above preparations, we can apply the Crandall-Liggett generation theorem \cite[Theorem~I]{CraLig71} to prove the existence of a $L_1$-global mild solution to the following $\omega$-fractional porous medium equation: 
\begin{equation}
\label{S5: FPME-omega}
\left\{\begin{aligned}
\partial_t u +(\omega-\Delta_g)^\sigma (|u|^{m-1}u  )&=0   &&\text{on}&&\M\times (0,\infty);\\
u(0)&=u_0    &&\text{on}&&\M .
\end{aligned}\right.
\end{equation}

Mild solutions are defined as the limit of a sequence of approximation solutions by implicit time discretization. 
More precisely, since a mild solution  does not depend on  a  specific choice of discretization of a time interval $[0,T]$, we  may equally divide $[0,T]$ into $n$ subintervals. Let  $t_k=kT/n$ for $k=0,1,\cdots,n$. Then the discretized problem to \eqref{S5: FPME-omega} is 
\begin{equation}
\label{S4: DPME}
\left\{\begin{aligned}
\frac{T}{n} (\omega-\Delta_g)^\sigma\Phi(u_{n,k;\omega}) &= u_{n,k-1;\omega}- u_{n,k;\omega};\\
u_{n,0;\omega} &=u_0  .
\end{aligned}\right.
\end{equation}
We define the piecewise solution as
$$
u_{n;\omega}(0)=u_0,\quad u_{n;\omega}(t)= u_{n,k;\omega} \quad \text{for } t\in (t_{k-1},t_k].
$$
Then the mild solution $u_\omega$ is defined as the uniform limit of $u_{n;\omega}$, i.e. for any $\varepsilon>0$, 
\begin{equation}
\label{S5: mild sol def and converg}
\|u_\omega(t)-u_{n;\omega}(t)\|_1 <\varepsilon, \quad t\in [0,T]
\end{equation}
for sufficiently large $n$. 
\cite[Theorem~I]{CraLig71} then implies that  the following proposition holds true.
\begin{prop}
\label{S6: global sol}
Let $(\M,g)$ be  an $(n+1)$-dimensional compact conical manifold. 
Assume that $m>0$. 
For every $u_0\in L_1(\M)$, \eqref{S5: FPME-omega} has a unique global mild solution 
$$
u_\omega\in C([0,\infty), L_1(\M)).
$$	
\end{prop}


In the following, we  will  show that   \eqref{S1: FPME} has a global weak solution   as long as $u_0\in L_\infty(\M) \xhookrightarrow{d} L_1(\M) $. 
\begin{definition}\label{Def: weak sol}
For $\omega\in [0,\infty)$, we say that $u$ is an $L_1$-weak solution to \eqref{S5: FPME-omega} on $[0,T)$ if
\begin{itemize}
\item $u\in L_\infty((0,T), L_{m+1}(\M))$ and,
\item $\Phi(u)\in L_2((0,T), D(( -\FD)^{\sigma/2})) \cap L_{\infty,loc}((0,T), D(( -\FD)^{\sigma/2}))$.
\end{itemize}
Moreover, for every $\phi\in C^1_0([0,T), C_c^\infty(\M))$  it holds that
\begin{equation}
\label{S4: weak sol}
\int_0^T \int_\M (\omega-\Delta_g)^{\sigma/2} \Phi(u)  (\omega-\Delta_g)^{\sigma/2}  \phi \, d\mu_g dt 
 =  \int^T_0\int_\M u \partial_t \phi \, d\mu_g dt+\int_\M  u_0\phi(0)\, d\mu_g.
\end{equation}
If in addition, $u$ satisfies
\begin{itemize}
\item $u\in C([0,T), L_1(\M))$, we call $u$ a weak solution to \eqref{S5: FPME-omega}.
\end{itemize}
\end{definition}

By the self-adjointness of $(\omega-\FD)^\sigma$, we have
\begin{equation}
\label{S5: Negative order Lap}
\langle(\omega-\Delta_g)^\sigma u ,v \rangle = \langle (\omega-\Delta_g)^{\sigma/2} u, (\omega -\Delta_g)^{\sigma/2} v \rangle   
\end{equation}
for all $u\in D(( -\FD)^{\sigma}), \, v\in D(( -\FD)^{\sigma/2})$.
In \eqref{S5: Negative order Lap}, $\omega=0$ is admissible.

Note that $(\omega-\underline{\Delta}_{F,1})^\sigma$ is an extension of $(\omega-\underline{\Delta}_F)^\sigma$ as $\underline{\Delta}_{F,1}$ extends $\underline{\Delta}_F$. 
From Proposition~\ref{S4: Lap-frac semigroup}, we infer that $u\in D((-\underline{\Delta}_{F,1})^\sigma)$ and $(\omega-\Delta_g)^\sigma u \in L_2(\M)$ implies
$$
u\in D(( -\underline{\Delta}_F)^\sigma).
$$
From Proposition~\ref{S4: semilinear thm}, we already know that 
$\Phi(u_{n,k;\omega})\in D(( -\underline{\Delta}_{F,1})^\sigma)$.
In addition, \cite[Proposition~4]{BreStr73} implies that  
\begin{equation}
\label{S5: L_infty contraction dis sol}
 \|u_{n,k;\omega}\|_\infty \leq 	\|u_0\|_\infty.
\end{equation}
In view of \eqref{S4: DPME}, \eqref{S5: L_infty contraction dis sol} reveals that $(\omega-\Delta_g)^\sigma \Phi(u_{n,k;\omega}) \in L_\infty(\M)$.

Multiplying the first line of \eqref{S4: DPME} by $\Phi(u_{n,k;\omega})$ and integrating over $\M$, we obtain in virtue of \eqref{S5: Negative order Lap} that
\begin{align}
\label{S4: weak est 1}
\notag&\frac{T}{n} \int_\M |(\omega-\Delta_g)^{\sigma/2} \Phi(u_{n,k;\omega})|^2\, d\mu_g\\
\notag =& \int_\M  u_{n,k-1;\omega} \Phi(u_{n,k;\omega})\, d\mu_g -\int_\M  u_{n,k;\omega} \Phi(u_{n,k;\omega})\, d\mu_g \\
\leq & (\int_\M  |u_{n,k-1;\omega}|^{m+1}  \, d\mu_g  )^{\frac{1}{m+1}} (\int_\M | \Phi(u_{n,k;\omega})|^{\frac{m+1}{m}}\, d\mu_g )^{\frac{m}{m+1}}   -  \int_\M | u_{n,k;\omega}|^{m+1} \, d\mu_g \\
\label{S4: weak est last}
=& \frac{1}{m+1}\int_\M  |u_{n,k-1;\omega}|^{m+1}  \, d\mu_g  + \frac{m}{m+1} \int_\M | u_{n,k;\omega}|^{m+1}\, d\mu_g -  \int_\M | u_{n,k;\omega}|^{m+1} \, d\mu_g\\
\notag \leq & \frac{1}{m+1} (\int_\M | u_{n,k-1;\omega}|^{m+1} \, d\mu_g  - \int_\M | u_{n,k;\omega}|^{m+1} \, d\mu_g ).
\end{align}
We have used the H\"older inequality in  \eqref{S4: weak est 1} and the Young's inequality in \eqref{S4: weak est last}. 
Then adding from $k=1$ to $k=n$ yields
\begin{equation}
\label{S4: unif est 1}
\int_0^T \int_\M |(\omega-\Delta_g)^{\sigma/2} \Phi(u_{n;\omega})|^2\, d\mu_g dt \leq \frac{1}{m+1}\int_\M |u_0 |^{m+1}\, d\mu_g .
\end{equation}
Note that $(\omega-\Delta_g)^{\sigma/2}$ is invertible.
Therefore, $ \Phi(u_{n;\omega})$ is uniformly bounded for all $n$ in $L_2((0,T), D(( -\Delta_g)^{\sigma/2}))$.

On the other hand, \eqref{S4: weak est 1} or \cite[Proposition~4]{BreStr73} also implies that
\begin{equation}
\label{S4: unif est 0}
\int_\M | u_{n,k;\omega}|^{m+1} \, d\mu_g \leq \int_\M | u_{n,k-1;\omega}|^{m+1} \, d\mu_g \leq \int_\M |u_0 |^{m+1}\, d\mu_g.
\end{equation}
This yields the uniform boundedness of $u_{n,k;\omega}(t)$ in $L_{m+1}(\M)$ for all $t\in [0,T)$,
which implies that 
\begin{equation}
\label{S4: unif est 2}
u_\omega\in L_\infty((0,T),L_{m+1}(\M) ).
\end{equation}
By the definition of mild solutions, for all $t\in [0,T)$, (up to a subsequence) $u_{n;\omega}(t)$ converge to $u_\omega(t)$ pointwise a.e. 
Hence we can conclude from \eqref{S4: unif est 1} that 
$$
(\omega-\Delta_g)^{\sigma/2} \Phi(u_{n;\omega})\rightharpoonup (\omega-\Delta_g)^{\sigma/2} \Phi(u_\omega)\quad \text{ in}\quad L_2((0,T), L_2(\M))
$$
and 
\begin{equation}
\label{S4: unif est 3}
\|(\omega-\Delta_g)^{\sigma/2} \Phi(u_{ \omega})\|_{L_2((0,T), L_2(\M))} \leq \frac{1}{m+1}\int_\M |u_0 |^{m+1}\, d\mu_g .
\end{equation}
Now multiplying the first line of \eqref{S4: DPME} by   $\phi\in C^1_0([0,T), C^\infty_c(\M))$  and integrating over $\M$, we infer that
$$
\langle (\omega-\Delta_g)^{\sigma/2}  \Phi(u_{n,k;\omega}), (\omega-\Delta_g)^{\sigma/2}  \phi \rangle =\frac{n}{T}\int_\M (u_{n,k-1;\omega}-u_{n,k;\omega})\phi\, d\mu_g.
$$
Then integrate over $[t_{k-1},t_k)$ and sum over $k=1,2,\cdots,n$. The right hand side equals
\begin{align*}
&\frac{n}{T}\sum\limits_{k=1}^n\int_{t_{k-1}}^{t_k}\int_\M (u_{n,k-1;\omega} -u_{n,k;\omega})\phi \, d\mu_g dt\\
=& \int_0^T\int_\M u_{n;\omega}(t) \frac{\phi(t+T/n) -\phi(t)}{T/n}\, d\mu_g dt +  \frac{n}{T}\int_0^{t_1} \int_\M  u_0 \phi(t)\, d\mu_g dt\\
& - \frac{n}{T}\int_{t_{n-1}}^T \int_\M u_{n;\omega}(T) \phi(t)\, d\mu_g dt.
\end{align*}
As  $(\omega-\Delta_g)^{\sigma/2}  \phi\in L^2((0,T), L_2(\M))$, letting $n\to\infty$ yields 
\begin{align}
\label{S4: weak sol-u_omega}
\notag & \int^T_0\int_\M u_\omega \partial_t \phi \, d\mu_g dt+\int_\M  u_0 \phi(0)\, d\mu_g\\
\notag =&\int_0^T \int_\M (\omega-\Delta_g)^{\sigma/2} \Phi(u_\omega)  (\omega-\Delta_g)^{\sigma/2}  \phi \, d\mu_g dt \\
 =&\int_0^T \int_\M   \Phi(u_\omega)  (\omega-\Delta_g)^\sigma   \phi \, d\mu_g dt.
\end{align}
The last equality will be used to obtain a weak solution to \eqref{S1: FPME}.


For every $v\in L_2(\M)$, it follows from Proposition~\ref{S4: contraction semigroup-Lp} and the formula above \cite[Formula~2.2.24]{Lunar95} that
\begin{align}
\label{S4: unif est 4}
\notag \|(\omega-\Delta_g)^{-\sigma} v\|_2
&\leq C \int_0^\infty t^{\sigma-1} \| e^{t(\Delta_g-\omega)} v\|_2\, dt\\
&\leq C \int_0^\infty t^{\sigma-1}  e^{-\omega t} \, dt \|v\|_2 \leq C \omega^{-\sigma}\|v\|_2.
\end{align}
This result, together with \eqref{S4: unif est 3}  and  \eqref{S4: control diff},  yields
\begin{align}
\label{S4: unif est 5}
\notag &\|(-\Delta_g)^{\sigma/2} \Phi(u_{ \omega})\|_{L_2((0,T), L_2(\M))} \\
\notag \leq &\|(\omega-\Delta_g)^{\sigma/2} \Phi(u_{ \omega})\|_{L_2((0,T), L_2(\M))} + C\omega^{\sigma/2} \|\Phi(u_\omega)\|_{L_2((0,T), L_2(\M))}\\
 \leq & C\|(\omega-\Delta_g)^{\sigma/2} \Phi(u_{ \omega})\|_{L_2((0,T), L_2(\M))}<\infty.
\end{align}
\eqref{S4: unif est 3} and \eqref{S4: unif est 5} imply that there exists a sequence $(\omega_k)_k$ with $\omega_k\to 0^+$ such that
\begin{equation}
\label{S4:converg-1}
(-\Delta_g)^{\sigma/2}\Phi(u_{\omega_k})\rightharpoonup w \quad \text{in}\quad L_2((0,T), L_2(\M))
\end{equation}
for some $w$ in $L_2((0,T), L_2(\M) )$;
and it follows from \eqref{S4: unif est 0} and \eqref{S4: unif est 2} that for a.a. $t\in (0,T)$
\begin{equation}
\label{S4:converg-2}
u_{\omega_k}(t) \rightharpoonup u(t)  \quad \text{in}\quad L_{ m+1 }(\M)
\end{equation}
for some $u\in L_\infty((0,T),L_{ m+1 }(\M))$.

Let us now refine the estimate for $\|(-\Delta_g)^{\sigma/2} \Phi(u_\omega)\|_2$. 
Suppose that $u_\omega,\hat{u}_\omega$ are mild  solutions to \eqref{S5: FPME-omega} with respect to the initial data $u_0,\hat{u}_0$.
It follows from \cite[Theorem~1]{BreStr73} that
$$
\| u_{n;\omega}(t)- \hat{u}_{n;\omega}(t)\|_1 \leq \| u_0- \hat{u}_0\|_1.
$$
Assume, to the contrary,  that
$$
\| u_0- \hat{u}_0\|_1<\| u_\omega(t)- \hat{u}_\omega (t)\|_1
$$
for some $t>0$. Then
\begin{align*}
\| u_0- \hat{u}_0 \|_1 &< \| u_\omega(t)- \hat{u}_\omega (t)\|_1\\
&= \int_\M [ u_\omega(t)- \hat{u}_\omega (t)] {\rm sign}(u_\omega(t)- \hat{u}_\omega (t))\, d\mu_g\\
&= \lim\limits_{n\to \infty} \int_\M [ u_{n;\omega}(t)- \hat{u}_{n;\omega} (t)] {\rm sign}(u_\omega(t)- \hat{u}_\omega (t))\, d\mu_g\\
&\leq \lim\limits_{n\to \infty} \| u_{n;\omega}(t)- \hat{u}_{n;\omega}(t)\|_1 .
\end{align*}
A  contradiction. Therefore, 
\begin{equation}
\label{S5: L1-contraction: L1 weak sol}
\| u_\omega(t)- \hat{u}_\omega(t)\|_1 \leq \| u_0- \hat{u}_0\|_1.
\end{equation}
When $m\neq 1$,  following the proof of \cite[Theorem 1]{BenCra81}, see also \cite[Proposition~8.1]{PalRodVaz12}, and using \eqref{S5: L1-contraction: L1 weak sol},  it holds that as $h\to 0$
\begin{equation}
\label{S4: der est for u_omega}
\frac{1}{h}\int_\M |u_\omega(t+h)-u_\omega(t)|\, d\mu_g \leq \frac{2}{|m-1|t} \|u_0\|_1 + o(1).
\end{equation}
Fix $\tau\in (0,T)$. For every $t\in (\tau,T)$, choose $h>0$ so small that $t+h<T$. 
For a time discretization of $[0,T)$ with $n$ subintervals $[t_{k-1},t_k)$, without loss of generality, we may always assume that  $t=t_i$ and $t+h=t_j$ for some $i,j\in \{1,2,\cdots, n-1\}$.
Otherwise, we can always change the end points of the closest subintervals, as the solution $u_\omega$ does not depend on the choice of time discretization.
Using \eqref{S4: weak est 1}, we have
\begin{align*}
 &(t_k-t_{k-1}) \int_\M |(\omega-\Delta_g)^{\sigma/2} \Phi(u_{n,k;\omega})|^2\, d\mu_g\\
\leq &  \frac{1}{m+1} \int_\M  ( |u_{n,k-1;\omega}|^{m+1} -|u_{n,k;\omega}|^{m+1}  )\, d\mu_g.
\end{align*}
Then we sum over all the subintervals $[t_{k-1}, t_k)$ contained in $[t,t+h)$.
\begin{align}
\notag &   \int_t^{t+h} \int_\M |(\omega-\Delta_g)^{\sigma/2} \Phi(u_{n;\omega}(s))|^2\, d\mu_g\, ds\\
\notag  \leq & \frac{1}{m+1} \int_\M  ( |u_{n;\omega}(t)|^{m+1} -|u_{n ;\omega}(t+h)|^{m+1}  )\, d\mu_g\\
\notag \leq & \frac{1}{m+1} \int_\M  \big| |u_{n;\omega}(t)|^m u_{n;\omega}(t) -|u_{n ;\omega}(t+h)|^m  u_{n ;\omega} (t+h) \big| \, d\mu_g \\
\label{S4: weak est 2}
\leq & C(u_0)     \frac{1}{m+1} \int_\M |     u_{n ;\omega}(t)-  u_{n ;\omega}(t+h) | \, d\mu_g   
\end{align}
for some $C=C(u_0)>0$. Here \eqref{S4: weak est 2} is due to \eqref{S5: L_infty contraction dis sol} and the fact that 
$$[x\to |x|^m x]\in C^{1-}(\R)\quad \text{for}\quad m>0.$$
Dividing both sides by $h$ and letting $n\to\infty$ yields that  
\begin{align}
\label{S4: weak est 3}
\notag &   \frac{1}{h}\int_t^{t+h} \int_\M |(\omega-\Delta_g)^{\sigma/2} \Phi(u_\omega)(s)|^2\, d\mu_g\, dt\\
\leq &      \frac{C(u_0)}{m+1}  \frac{1}{h}\int_\M |     u_\omega(t+h)-  u_\omega(t)| \, d\mu_g\\
\label{S4: weak est 4}
\leq & C(u_0) \frac{2 }{ (m+1)|m-1|\tau }   \|u_0\|_1 +o(1).
\end{align}
Here \eqref{S4: weak est 3} follows directly from \eqref{S5: mild sol def and converg}  and \eqref{S4: weak est 2},
and    \eqref{S4: weak est 4} is derived from \eqref{S4: der est for u_omega}.

The above inequality holds for all $h>0$ small.
Therefore, for  a.a. $t\in (\tau,T)$, $\|(\omega-\Delta_g)^{\sigma/2} \Phi(u_\omega)(t)\|_2$ is uniformly bounded with bound independent of $\omega$, i.e.
\begin{equation*}
\|(\omega-\Delta_g)^{\sigma/2} \Phi(u_\omega)(t)\|_2 <M ,\quad t\in (\tau,T).
\end{equation*} 
In the sequel,  $M>0$ always denotes a constant independent of $\omega$.
By an analogous argument to \eqref{S4: unif est 5}, we can infer that
\begin{equation}
\label{S4: unif est 6}
\|( -\Delta_g)^{\sigma/2} \Phi(u_\omega)(t)\|_2 <M ,\quad t\in (\tau,T).
\end{equation}
Notice that \eqref{S5: mild sol def and converg} and \eqref{S5: L_infty contraction dis sol}  imply that
\begin{equation}
\label{S5: L_infty contraction mild sol}
\| u_\omega(t)\|_\infty \leq \|u_0\|_\infty .
\end{equation}
One can conclude that
$$
\|\Phi(u_\omega)\|_{L_\infty((\tau,T),L_2(\M) )}<M,
$$
and thus
$$
\|\Phi(u_\omega)\|_{L_\infty((\tau,T), D((-\FD )^{\sigma/2}))}<M .
$$
Hence $u_\omega$ is indeed a weak solution to \eqref{S5: FPME-omega} for any $\omega>0$. 

Now we can apply Proposition~\ref{S4: char domain of Delta-Lp} to show  that for every $t\in (\tau,T)$, there exists a subsequence of $(\omega_k)_k$, still denoted by $(\omega_k)_k$, such that
\begin{equation}
\label{S4:converg 3}
\Phi(u_{\omega_k})(t)\to v(t) \quad \text{in} \quad L_2(\M),
\end{equation}
and as $D((-\FD)^{\sigma/2})$ is a Hilbert space,
\begin{equation}
\label{S4:converg 4}
\Phi(u_{\omega_k})(t) \rightharpoonup v(t) \quad \text{in} \quad D((-\FD )^{\sigma/2}).
\end{equation}
We can thus infer that $\Phi(u_{\omega_k})(t) \to v(t)$ pointwise a.e. (up to a subsequence).
In view of \eqref{S4:converg-1} and \eqref{S4:converg-2}, we conclude that 
$v(t)=\Phi(u)(t)$ a.e.;  and since $\tau$ is arbitrary, in \eqref{S4:converg-1}, we have
 $w=(-\Delta_g)^{\sigma/2}\Phi(u)$.

Letting $\omega_k\to 0^+$ in \eqref{S4: weak sol-u_omega}, then \eqref{S5: Negative order Lap},  \eqref{S4:converg-2}  and the dominated convergence theorem imply that
\begin{align}
\notag  &\int^T_0\int_\M u \partial_t \phi \, d\mu_g dt+\int_\M u_0\phi(0)\, d\mu_g  \\
   =& \lim\limits_{k\to \infty} \int^T_0\int_\M u_{\omega_k} \partial_t \phi \, d\mu_g dt+\int_\M u_0\phi(0)\, d\mu_g\\
\notag  = &   \lim\limits_{k\to \infty}\int_0^T \int_\M   ( \omega_k-\Delta_g)^{\sigma/2 } \Phi(u_{\omega_k} )  ( \omega_k-\Delta_g)^{\sigma/2 }  \phi \, d\mu_g dt\\
\notag  = & \lim\limits_{k\to \infty} \int_0^T \int_\M    \Phi(u_{\omega_k} )  ( \omega_k-\Delta_g)^\sigma   \phi \, d\mu_g dt\\
\label{S4: u est 1}
= & \lim\limits_{k\to \infty}\int_0^T \int_\M    \Phi(u_{\omega_k} )  ( -\Delta_g)^\sigma   \phi \, d\mu_g dt\\
\label{S4: u est 2}
&+\lim\limits_{k\to \infty}\int_0^T \int_\M    \Phi(u_{\omega_k} )  [( \omega_k-\Delta_g)^\sigma - ( -\Delta_g)^\sigma ]  \phi \, d\mu_g dt
.
\end{align}
Since for any $1<p<\infty$,
$$
\|( \omega_k-\Delta_g)^\sigma  \phi  -  ( -\Delta_g)^\sigma  \phi\|_p \leq C\omega_k^\sigma \|\phi\|_p,
$$
in view of \eqref{S4: unif est 0}, 
we conclude that \eqref{S4: u est 2} tends to $0$ as $k\to \infty$. 
\eqref{S4: u est 1} can be estimated as follows.
\begin{align*}
&\lim\limits_{k\to \infty}\int_0^T \int_\M    \Phi(u_{\omega_k} )  ( -\Delta_g)^\sigma   \phi \, d\mu_g dt\\
=& \lim\limits_{k\to \infty}\int_0^T \int_\M    ( -\Delta_g)^{\sigma/2}   \Phi(u_{\omega_k} )  ( -\Delta_g)^{\sigma/2}   \phi \, d\mu_g dt \\
=&\int_0^T \int_\M    ( -\Delta_g)^{\sigma/2}   \Phi(u )  ( -\Delta_g)^{\sigma/2}   \phi \, d\mu_g dt.
\end{align*}
In the last step, we have used \eqref{S4:converg-1}.
This gives rise to
$$
 \int^T_0\int_\M u \partial_t \phi \, d\mu_g dt+\int_\M u_0\phi(0)\, d\mu_g
 =\int_0^T \int_\M   ( -\Delta_g)^{\sigma/2 } \Phi(u )  ( -\Delta_g)^{\sigma/2 }  \phi \, d\mu_g dt.
$$
Therefore, we obtain an $L_1$-weak solution to \eqref{S1: FPME}.


\begin{lem}\label{S5: Lemma-unique}
There is at most one $L_1$-weak solution to \eqref{S1: FPME} with $u_0\in L_\infty(\M)$.
\end{lem}
\begin{proof}
The proof is basically the same as the porous medium equation, c.f. \cite[Theorem~5.3]{Vaz07}.
First observe that the weak formulation~\eqref{S4: weak sol} with $\omega=0$ still holds true for functions $\phi\in W^1_2([0,T), D((-\underline{\Delta}_F)^{\sigma/2}))$ with $\phi(T)=0$. Indeed, take a sequence 
$$
C^1_0([0,T), C^\infty_c(\M))\ni \phi_n \to \phi \quad \text{in}\quad  W^1_2([0,T), D((-\underline{\Delta}_F)^{\sigma/2})).
$$
Then
$$
\int_0^T \int_\M ( -\Delta_g)^{\sigma/2} \Phi(u)  ( -\Delta_g)^{\sigma/2}  \phi_n \, d\mu_g dt 
 =  \int^T_0\int_\M u \partial_t \phi_n \, d\mu_g dt+\int_\M  u_0\phi_n(0)\, d\mu_g.
 $$
Letting $n\to \infty$ on both sides yields the desired result. 

Assume, to the contrary, there are two $L_1$-weak solutions $u, \tilde{u}$. Define 
$$
\phi(t)=\int_t^T (\Phi(u)-\Phi(\tilde{u}))\, ds ,\quad 0\leq t\leq T.
$$
Then $\phi\in W^1_2((0,T), D((-\underline{\Delta}_F)^{\sigma/2}))$ with $\phi(T)=0$ is a valid test function.
Multiplying $\phi$ to the first line of \eqref{S1: FPME} and integrating over $[0,T)\times \M$ yields 
\begin{align*}
\int_0^T\int_\M &(-\Delta_g)^{\sigma/2} (\Phi(u)-\Phi(\tilde{u}))(t) [\int_t^T (-\Delta_g)^{\sigma/2} (\Phi(u)-\Phi(\tilde{u}))(s)\, ds] \, d\mu_g dt\\
&+ \int_0^T\int_\M (u-\tilde{u})(t) (\Phi(u)-\Phi(\tilde{u}))(t)  \, d\mu_g dt=0.
\end{align*}
The first term on the left equals
$$
\frac{1}{2}\int_\M  | \int_0^T (-\Delta_g)^{\sigma/2} (\Phi(u)-\Phi(\tilde{u}))(t) \, dt|^2 \,   d\mu_g.
$$
Therefore, 
\begin{align*}
\frac{1}{2}\int_\M  &| \int_0^T (-\Delta_g)^{\sigma/2} (\Phi(u)-\Phi(\tilde{u}))(t) \, dt|^2  \, d\mu_g\\
&+ \int_0^T\int_\M (u-\tilde{u})(t) (\Phi(u)-\Phi(\tilde{u}))(t)  \, d\mu_g dt=0.
\end{align*}
Note that $\Phi$ is monotone. So both terms are non-negative. This implies that $u=\tilde{u}$ a.e..
\end{proof}


One can further argue that the solution $u\in C([0,T),L_1(\M))$. 
To this end, we will first prove the following two lemmas.
\begin{lem}\label{S5: Lem-weak cont of u}
For any $\sigma\in (0,1]$, assume that $u$ is the unique $L_1-$weak solution to \eqref{S1: FPME}. Then for any $\phi\in C^1_0([0,T), D((-\underline{\Delta}_F)^{\sigma/2}))$,
$$
\lim\limits_{t\to 0^+}\int_\M u(t) \phi (t)\, d\mu_g = \int_\M u_0 \phi(0) \, d\mu_g.
$$
\end{lem}
\begin{proof}
The proof of this lemma is partially based on Theorem~\ref{S2: other properties}(I) and (II) in Section~\ref{Section 6}.
Their proofs work for $L_1-$weak solutions and thus apply here.
Given any $\varepsilon\in (0,T)$ and $\phi\in C^1_0([0,T), D((-\underline{\Delta}_F)^{\sigma/2}))$, we consider the problem.
\begin{equation}
\label{S1: FPME-2}
\left\{\begin{aligned}
\partial_t v +(-\Delta_g)^\sigma \Phi(v)&=0   &&\text{on}&&\M\times (\varepsilon,T);\\
v(\varepsilon)&=u(\varepsilon)   &&\text{on}&&\M.
\end{aligned}\right.
\end{equation}
By  Theorem~\ref{S2: other properties}(I)  and the same process we used for \eqref{S1: FPME}, there exists a unique $L_1$-weak solution to \eqref{S1: FPME-2} such that
\begin{equation}
\label{S4: weak sol-2}
\int_\varepsilon^T \int_\M ( -\Delta_g)^{\sigma/2} \Phi(v)  ( -\Delta_g)^{\sigma/2}  \phi \, d\mu_g dt 
 =  \int^T_\varepsilon\int_\M v \partial_t \phi \, d\mu_g dt+\int_\M  u(\varepsilon)\phi(\varepsilon)\, d\mu_g.
\end{equation}
It follows from Lemma~\ref{S5: Lemma-unique} that $v(t)=u(t)$ a.e. for $t\in [\varepsilon,T)$. Because
\begin{align*}
 | \int^\varepsilon_0 \int_\M ( -\Delta_g)^{\sigma/2} \Phi(u)  ( -\Delta_g)^{\sigma/2}  \phi \, d\mu_g \, dt| 
\leq   C  \int^\varepsilon_0 \|( -\Delta_g)^{\sigma/2} \Phi(u)\|_2 \,  dt \to 0
\end{align*}
and
\begin{align*}
|\int^\varepsilon_0\int_\M u \partial_t \phi \, d\mu_g \, dt| \leq  C \int^\varepsilon_0 \|u(t)\|_2 \,  dt  \to 0
\end{align*}
as $\varepsilon\to 0$ due to Theorem~\ref{S2: other properties}(II),
the asserted results follows by letting $\varepsilon\to 0$ in \eqref{S4: weak sol-2}.
\end{proof}

Next we will show the $L_1$-contraction property of $L_1$-weak solutions to \eqref{S1: FPME}.
\begin{lem}\label{S7: Lem-L1-contraction}
Suppose that $u,\hat{u}$ are $L_1$-weak solutions to \eqref{S1: FPME} with respect to the initial data $u_0,\hat{u}_0$. Then for any $t\geq 0$
$$
\| u(t)- \hat{u}(t)\|_1 \leq \| u_0- \hat{u}_0\|_1.
$$
\end{lem}
\begin{proof}
This assertion follows by an analogous argument to \eqref{S5: L1-contraction: L1 weak sol}.
\end{proof}

Now we are ready to prove the continuity of $u$.
Take a sequence 
$$
D((-\underline{\Delta}_F)^{\sigma/2})\cap L_\infty(\M)\ni u_{0,n} \to u_0 \quad \text{in } L_1(\M).
$$
We denote by $u_n$ the unique $L_1$-weak solution to \eqref{S1: FPME} with initial data $u_{0,n}$. 

In particular, Lemma~\ref{S5: Lem-weak cont of u} shows that  
$$
\lim\limits_{t\to 0^+}\int_\M u_n(t) u_{0,n}  \, d\mu_g = \int_\M  u^2_{0,n} \, d\mu_g.
$$
Now we look at 
$$
\langle u_n(t)-u_{0,n}, u_n(t)-u_{0,n} \rangle= \|u_n(t)\|_2^2 + \|u_{0,n}\|_2^2 - 2\langle u_n(t) , u_{0,n} \rangle.
$$
Taking $\limsup\limits_{t\to 0^+}$ on both sides yields
$$
\limsup\limits_{t\to 0^+}\langle u_n(t)-u_{0,n}, u_n(t)-u_{0,n} \rangle= \limsup\limits_{t\to 0^+}\|u_n(t)\|_2^2 - \|u_{0,n}\|_2^2 \leq 0.
$$
The last inequality is due to Theorem~\ref{S2: other properties}(II). Note that the proof for Theorem~\ref{S2: other properties}(II) is completely independent of the convergence of the solutions at $t=0$.
This implies that
$$
\lim\limits_{t\to 0^+}\langle u_n(t)-u_{0,n}, u_n(t)-u_{0,n} \rangle= 0,
$$
which shows that $u_n$ is right continuous at  $t=0$  in the $L_2(\M)$-norm and thus in the $L_1(\M)$-norm.
Now we have
\begin{align}
\notag &\lim\limits_{t\to 0^+}\|u(t)-u_0\|_1 \\
\notag \leq & \lim\limits_{n\to \infty} \lim\limits_{t\to 0^+}(\|u(t)-u_n(t)\|_1 + \|u_n(t)-u_{0,n}\|_1 +\|u_{0,n}-u_0\|_1)\\
\label{S5: strong cont est 1}
\leq & 2 \lim\limits_{n\to \infty} \|u_{0,n}-u_0\|_1 +  \lim\limits_{n\to \infty} \lim\limits_{t\to 0^+}\|u_n(t)-u_{0,n}\|_1=0
\end{align}
The inequality in \eqref{S5: strong cont est 1} follows from Lemma~\ref{S7: Lem-L1-contraction}. 
This shows the right continuity of $u$ at  $t=0$  in  the $L_1(\M)$-norm. 
To show the interior temporal regularity, first note that for every $\varepsilon>0$, there exists a $\tau>0$ such that $0<h<\tau$ implies 
$$
\|u(h)-u(0)\|_1\leq \varepsilon.
$$
We apply Lemmas~\ref{S5: Lemma-unique} and \ref{S7: Lem-L1-contraction} again and obtain for every $t>0$ and $0<h<\tau$
$$
\|u(t+h)-u(t)\|_1 \leq \|u(h)-u(0)\|_1\leq \varepsilon.
$$
This shows that 
$
u\in C([0,T),L_1(\M)),
$
and thus $u$ is indeed a weak solution to \eqref{S1: FPME}.

\begin{theorem}
\label{S6: global weak sol}
Let $(\M,g)$ be  an $(n+1)$-dimensional compact conical manifold. 
Assume that $m>0$ with $m\neq 1$. 
Then for any $T>0$ and every $u_0\in  L_\infty(\M)$, \eqref{S1: FPME} has a unique  weak solution on the  interval $[0,T)$. Moreover, the solution depends continuously on $u_0$ in the norm $C([0,T), L_1(\M))$.
\end{theorem}
\begin{proof}
The continuous dependence of the weak solution on $u_0$ is a direct consequence of Lemma~\ref{S7: Lem-L1-contraction}.
\end{proof}
 
\begin{remark}
\begin{itemize}
\item[] 
\item[(i)] Well-posedness for $L_2$-initial data is related to the smoothing effect, which can be established via some nonlocal logarithmic Sobolev inequality based on  \eqref{S3.4: Sobolev-Poincare}. The idea follows from a generalization of the argument in \cite{RoidosShaoPre1}, which will be discussed in a future work.
\item[(ii)]The method used in this section can be applied to the (fractional) porous medium equation on general complete Riemannian manifolds without any difficulty to establish the uniqueness and existence of global weak solution. See \cite[Theorem~2.3]{RoidosShaoPre1}.
\end{itemize}
\end{remark}


\section{Some Other Properties of the Fractional Porous Medium Equation}\label{Section 6}

\begin{theorem}
\label{S2: other properties}
Under the same assumptions as in Theorem~\ref{S6: global weak sol}, the weak solutions to \eqref{S1: FPME} satisfy the following properties.
\begin{itemize}
\item[(I) ] {\em Comparison principle: }If $u, \hat{u}$ are the unique weak solutions to \eqref{S1: FPME} with initial data $u_0, \hat{u}_0$, respectively, then $u_0\leq \hat{u}_0$ a.e. implies $u \leq \hat{u}$ a.e.
\item[(II) ] {\em $L_p$-contraction: }For all $0 \leq t_1 \leq t_2$ and $1\leq p\leq \infty$
$$
\|u(t_2)\|_p \leq \|u(t_1)\|_p.
$$
\item[(III) ] {\em Conservation of mass: } For  all $t\geq 0$, it holds  that
$$
\int_\M u(t) \, d\mu_g = \int_\M u_0 \, d\mu_g.
$$
\end{itemize}
\end{theorem}
\begin{proof}
(I) Suppose $u_{n,k;\omega}$ and $\hat{u}_{n,k;\omega}$ are the solutions obtained in the $k$-th step of the discretized problem~\eqref{S4: DPME} with respect to the initial   data  $u_0, \hat{u}_0$, respectively.
Then \cite[Proposition~5]{BreStr73} implies that
\begin{equation}
\label{S4: contraction PME}
\int_\M [u_{n,k+1;\omega}- \hat{u}_{n,k+1;\omega}]^+ \, d\mu_g \leq \int_\M [u_{n,k;\omega}- \hat{u}_{n,k;\omega}]^+ \, d\mu_g \leq \int_\M [u_0- \hat{u}_0]^+ \, d\mu_g.
\end{equation}
We immediately conclude from the above inequality that $u_0\leq \hat{u}_0$ a.e. implies $u_{n;\omega} \leq \hat{u}_{n,\omega}$ a.e. 

Suppose that, for some $t$, $\mathsf{m}(\{u_\omega(t)>\hat{u}_\omega(t)\})>0$, where $\mathsf{m}$ is the measure on $(\M,g)$. Pick $f\in L_\infty(\M)$ such that $f=1$ on $\{u_\omega(t)>\hat{u}_\omega(t)\}$, and $f=0$ elsewhere. Then
\begin{align*}
0<\int_\M [u_\omega (t) - \hat{u}_\omega (t)]^+ \, d\mu_g &=\int_\M (u_\omega (t)- \hat{u}_\omega (t))f \, d\mu_g\\
= & \lim\limits_{n\to \infty} \int_\M (u_{n;\omega} (t)- \hat{u}_{n;\omega} (t))f \, d\mu_g \\
\leq & \lim\limits_{n\to \infty} \int_\M [u_{n;\omega} (t)- \hat{u}_{n;\omega} (t)]^+ \, d\mu_g\\
\leq & \int_\M [u_0- \hat{u}_0]^+ \, d\mu_g=0,
\end{align*}
as $u_{n;\omega} (t)- \hat{u}_{n;\omega} (t) \to u_\omega (t) - \hat{u}_\omega (t)$ in $L_1(\M)$. A  contradiction. Applying the same argument to $u(t)$ and $\hat{u}(t)$ yields the desired result.

(II) By a similar argument to part I, we have 
$$
\int_\M [u(t_2)- \hat{u}(t_2)]^+ \, d\mu_g \leq \int_\M [u(t_1)- \hat{u}(t_1)]^+ \, d\mu_g,\quad t_2>t_1.
$$
This readily  implies that 
$$
\|u(t_2)\|_1 \leq 	\|u(t_1)\|_1 \quad \text{and} \quad \|u(t_2)\|_\infty \leq 	\|u(t_1)\|_\infty.
$$
Hence for the remaining case $1<p<\infty$, it suffices to consider $t_1=0$, for otherwise we can replace $u_0$ by $u(t_1)$ in \eqref{S1: FPME}.
By \cite[Proposition~4]{BreStr73}, the solution of the discretized problem~\eqref{S4: DPME} satisfies
$$
\|u_{n,k+1;\omega}\|_p \leq \|u_{n,k;\omega}\|_p \leq \| u_0\|_p,\quad 1\leq p\leq \infty.
$$
Taking limit in the above inequality yields (II) when $1<p<\infty$.


(III) Since constant functions belong to $D(\FD)$ and thus are in $D((-\FD)^\sigma)$, we can take $\phi\equiv 1$ as a test function in \eqref{S4: weak sol} with $\omega=0$, i.e.
\begin{align*}
\int_\M [u(t_2)-u(t_1)] \phi \, d\mu_g = - \int_{t_1}^{t_2} \int_\M (-\Delta_g)^{\sigma/2} \Phi(u ) ( -\Delta_g)^{\sigma/2} \phi  \, d\mu_g\, dt=0.
\end{align*}
This proves the conservation of mass.
\end{proof}

\begin{remark}
The conservation of mass holds for all $m>0$, and  this is an essential difference from the extinction phenomenon observed for the (fractional) porous medium equation on $\R^N$ and in bounded domains  with vanishing Dirichlet boundary condition, cf. \cite{BenCra8102} for the finite time extinction of solutions for the porous medium equation and \cite{PalRodVaz12} for the fractional porous medium equation.
\end{remark}

\begin{remark}
The proof of conservation of mass actually show that if $(\M,g)$ is an anti-cone then mass is actually conserved on both components of $\M$. This means the nonlocal diffusion of \eqref{S1: FPME} cannot transport any mass from one component to another through the conical ends.
\end{remark}

\section{Strong Solutions}\label{Section 7}

\begin{definition}\label{Def: Strong Sol}
We call a weak solution $u$  to \eqref{S1: FPME} on $[0,T)$ a strong solution if in addition
$$
\partial_t u \in L_{\infty,loc}((0,T), L_1(\M)))
$$
and
$$
(-\Delta_g)^\sigma \Phi(u) \in L_{\infty,loc}((0,T),L_1(\M)).
$$
\end{definition}

We will use a modification of the method in \cite{PalRodVaz11,PalRodVaz12} to prove that     solutions to \eqref{S1: FPME} are indeed strong solutions.

We first consider the case $m\neq 1$.
By \cite[Theorems 1 and 2]{BenCra81}, we have
\begin{itemize}
\item[(i)] $\partial_t u $ is a Radon measure;
\item[(ii)] $\displaystyle \limsup\limits_{h\to 0^+} \frac{\| u(t+h)-u(t)\|_1}{h}\leq \frac{2}{|m-1|t}\|u_0\|_1$.
\end{itemize}

For any $f\in L_{1,loc}(0,T)$, define the Steklov average of $f$ by
$$
f^h(t):=\frac{1}{h}\int^{t+h}_t f(s) \, ds,
$$
and 
$$
\delta^h f (t):=\partial_t f^h(t)= \frac{f(t+h)-f(t)}{h}\quad \text{a.e.}
$$
We can write the weak formulation \eqref{S4: weak sol} with $\omega=0$ as
\begin{equation}
\label{S3: weak sol2}
\int_0^T \int_\M (\delta^h u )\phi \, d\mu_g\,dt 
+ \int_0^T \int_\M  (-\Delta_g)^{\sigma/2} (\Phi(u))^h (-\Delta_g)^{\sigma/2} \phi \, d\mu_g\,dt =0.
\end{equation}

For any $[\tau,S]\subset (0,T)$, we choose $\zeta\in C_0^1((0,T),[0,1])$ such that $\zeta\equiv 1$ on $[\tau,S]$ and vanishes outside $[\tau^\prime, S^\prime]$ for some $[\tau^\prime, S^\prime]\subset (0,T)$ with $[\tau,S]\subset (\tau^\prime, S^\prime)$.
Let us take $\phi=\zeta \delta^h (\Phi(u))$.
Then \eqref{S3: weak sol2} yields
\begin{align}
\label{S3: weak sol3}
\notag &\quad\int_0^T \int_\M \zeta (\delta^h u) \delta^h (\Phi(u)) \, d\mu_g\,dt \\
&+ \int_0^T \int_\M \zeta  (-\Delta_g)^{\sigma/2} (\Phi(u))^h (-\Delta_g)^{\sigma/2} \partial_t(\Phi(u))^h \, d\mu_g\,dt
=0.
\end{align}
Since  $(\delta^h u)(\delta^h \Phi(u))\geq c (\delta^h (|u|^{(m-1)/2}u))^2$, cf. \cite[Section~5.3]{PalRodVaz11},
the first term on the left hand side of \eqref{S3: weak sol3} satisfies that
$$
\int_0^T \int_\M \zeta (\delta^h u ) \delta^h (\Phi(u)) \, d\mu_g\,dt \geq c \int_0^T \int_\M \zeta (\delta^h (|u|^{(m-1)/2}u))^2 \, d\mu_g\,dt.
$$
The second term on the left hand side of \eqref{S3: weak sol3} can be estimated as follows.
\begin{align}
\notag & |\int_0^T \int_\M \zeta  (-\Delta_g)^{\sigma/2} (\Phi(u))^h (-\Delta_g)^{\sigma/2} \partial_t(\Phi(u))^h \, d\mu_g\,dt | \\
\notag\leq &   \frac{1}{2}\int_{\tau^\prime}^{S^\prime} \int_\M |\zeta^\prime | [ (-\Delta_g)^{\sigma/2} (\Phi(u))^h ]^2\, d\mu_g\,dt\\
\label{S7: est for strong sol}
\leq & C \int_{\tau^\prime}^{S^\prime} \int_\M   [ (-\Delta_g)^{\sigma/2} (\Phi(u))^h ]^2\, d\mu_g\,dt  .
\end{align}
Recall a weak solution $\Phi(u)\in L_{\infty,loc}((0,T), D((-\underline{\Delta}_F)^{\sigma/2}))$. By Minkowski's inequality, 
\begin{align*}
(\int_\M   [ (-\Delta_g)^{\sigma/2} (\Phi(u))^h ]^2\, d\mu_g )^{1/2}
& = \frac{1}{h } (\int_\M   (\int_t^{t+h} (-\Delta_g)^{\sigma/2} \Phi(u)(s)\,ds)^2\, d\mu_g)^{1/2}\\
&\leq \frac{1}{h } \int_t^{t+h} (\int_\M [ (-\Delta_g)^{\sigma/2} \Phi(u)(s)]^2\, d\mu_g)^{1/2}     \, ds\\
&\leq \sup_{t\in [\tau^\prime, S^\prime]}(\int_\M [ (-\Delta_g)^{\sigma/2} \Phi(u)(t)]^2\, d\mu_g)^{1/2} .
\end{align*}
Thus \eqref{S7: est for strong sol} is uniformly bounded in $h$. 
This implies that  
$$
\| \delta^h (|u|^{(m-1)/2}u)\|_{L_2([\tau,S],L_2(\M)) }
$$ 
is uniformly bounded in $h$,
and thus
$$
\partial_t (|u|^{(m-1)/2}u)\in L_{2,loc}((0,T), L_2(\M) ).
$$
Since (i) implies $u\in BV((\tau,T), L_1(\M))$ for any $\tau>0$, it then follows from \cite[Theorem~1.1]{BenGar95} that
$$
\partial_t u \in L_{\infty,loc}((0,T), L_1(\M)) \quad \text{with} \quad \|\partial_t u(t))\|_1 \leq \frac{2}{|m-1|t}\|u_0\|_1.
$$
The last inequality is derived from (ii). Since $(-\Delta_g)^\sigma \Phi(u)= -\partial_t u$ in the distributional sense, the above conclusion reveals that
$$
(-\Delta_g)^\sigma \Phi(u) \in L_{\infty,loc}((0,T),L_1(\M)).
$$
\begin{theorem}
\label{S7: Thm: strong solution}
Let $(\M,g)$ be  an $(n+1)$-dimensional compact conical manifold. 
Assume that $m>0$ with $m\neq 1$ and $  u_0\in  L_\infty(\M)$. 
Then for any $T>0$, \eqref{S1: FPME} with $\sigma\in (0,1)$ has a unique strong solution on $[0,T)$.
\end{theorem}

For the remaining case $m=1$, by \cite[Theorem 4.1.4]{Pazy} and Proposition~\ref{S4: Lap-frac semigroup}, we obtain the following theorem.
\begin{theorem}
\label{S7: Thm: strong solution m=1}
Let $(\M,g)$ be  an $(n+1)$-dimensional compact conical manifold. 
Assume that $m=1$ and $1<p<\infty$. 
Then for any $T>0$ and every $u_0\in L_p(\M)$, \eqref{S1: FPME} with $\sigma\in (0,1)$ has a unique  solution $u$ on $[0,T)$ such that
\begin{align*}
u\in & C^1([0,T),L_p(\M))\cap C([0,T),D((-\underline{\Delta}_{F,p})^\sigma))
\end{align*}
When $u_0\in L_\infty(\M)$, this solution is strong, depends continuously on $u_0$ in the the norm $C([0,T),L_1(\M))$ and satisfies (I)-(III) of Theorem~\ref{S2: other properties}.
\end{theorem}
\begin{proof}
The unique solution to \eqref{S1: FPME} when $m=1$ is given by 
$$
u(t)=e^{-t(-\Delta_g)^\sigma} u_0.
$$
Then the uniqueness, regularity and continuous dependence of $u$ on $u_0$ follows immediately. (I) and (II) of Theorem~\ref{S2: other properties} is a direct consequence of Proposition~\ref{S4: Lap-frac semigroup}. (III) of Theorem~\ref{S2: other properties} is obtained by multiplying both sides of \eqref{S1: FPME} by $\phi\equiv 1$ and integrating over $[t_1, t_2] \times \M$.
\end{proof}


\section{The Porous Medium Equation on conical Manifolds}\label{Section 8}

When $\sigma=1$, \eqref{S1: FPME} reduces to the usual porous medium equation
\begin{equation}
\label{S8: PME}
\left\{\begin{aligned}
\partial_t u - \Delta_g (|u|^{m-1}u  )&=0   &&\text{on}&&\M\times (0,\infty);\\
u(0)&=u_0   &&\text{on}&&\M.
\end{aligned}\right.
\end{equation}
In this last section, we will study the well-posedness and various properties of \eqref{S8: PME} for $m>0$. These results generalize the study of the porous medium equation on manifolds with singularities in \cite{RoiSch15, RoiSch17, Shao16}.

Since the proofs for the  results in this section are analogous to the fractional porous medium equation, only necessary modifications of the proofs will be presented.

\begin{theorem}
\label{Thm: Porous Medium}
Let $(\M,g)$ be  an $(n+1)$-dimensional compact conical manifold. 
Assume that $m>0$ with $m\neq 1$.  
Then for any $T>0$ and every $ u_0\in  L_\infty(\M)$, \eqref{S8: PME} has a unique strong solution on $[0,T)$ in the sense that
\begin{itemize}
\item[(i)] $u\in L_\infty((0,T), L_{m+1}(\M))$, and 
\item[(ii)] $\Phi(u)\in L_2((0,T), H^1(\M)) \cap L_{\infty,loc}((0,T), H^1(\M))$, and
\item[(iii)] $u\in C([0,T), L_1(\M))$, and
\end{itemize}
for every $\phi\in C^1_0([0,T), C_c^\infty(\M))$, it holds that
\begin{equation}
\label{S8: weak sol of PME}
\int_0^T   \langle \nabla \Phi(u) ,  \nabla  \phi \rangle_\mathcal{T} \,   dt 
 =  \int^T_0\int_\M u \partial_t \phi \, d\mu_g\,  dt+\int_\M  u_0\phi(0)\, d\mu_g.
\end{equation}
In addition,    
\begin{equation}
\label{S8: classic sol of PME 1}
\partial_t u \in L_{\infty,loc}((0,T), L_1(\M))
\end{equation}
and
\begin{equation}
\label{S8: classic sol of PME 2}
 \Delta_g \Phi(u) \in L_{\infty,loc}((0,T),L_1(\M)).
\end{equation}
Moreover the solutions satisfy the following properties.
\begin{enumerate}
\item  {\em Continuous dependence on the  initial data:} The solution depends continuously on $u_0$ in the norm $C([0,T), L_1(\M))$.
\item  {\em Comparison principle:} If $u, \hat{u}$ are the unique strong solutions to \eqref{S8: PME}  with initial data $u_0, \hat{u}_0$, respectively, then $u_0\leq \hat{u}_0$ a.e. implies $u \leq \hat{u}$ a.e.
\item {\em $L_p$-contraction:} For all $0 \leq t_1 \leq t_2$ and $1\leq p\leq \infty$
$$
\|u(t_2)\|_p \leq \|u(t_1)\|_p.
$$
\item  {\em Conservation of mass:} For  all $t\geq 0$, it holds  that
$$
\int_\M u(t) \, d\mu_g = \int_\M u_0 \, d\mu_g.
$$
\end{enumerate}
\end{theorem}
\begin{proof}
Note that by the divergence theorem
$$
\langle -\Delta_g  u ,v \rangle = \langle \nabla u, \nabla v \rangle_\mathcal{T}   
$$
for all $u\in D( \FD ), \, v\in W^1_2(\M)$.
So following the same approach for the fractional porous medium equation in Section~\ref{Section 5}, one can show, with $(-\Delta_g)^{\sigma/2}$ being  replaced by $\nabla$, that \eqref{S8: PME}  has a unique solution satisfying (i)-(iii) and \eqref{S8: weak sol of PME}.

\eqref{S8: classic sol of PME 1} can be proved along the same line of argument in Section~\ref{Section 7}. Once we have \eqref{S8: classic sol of PME 1}, note that \eqref{S8: weak sol of PME} implies that 
$$
\int^S_\tau\int_\M  \partial_t u \phi \, d\mu_g \, dt+ \int_\tau^S   \langle \nabla \Phi(u) ,  \nabla  \phi \rangle_\mathcal{T} \,   dt 
 =  0 
$$
for any $\tau,S\in [0,T)$
Thus $-\partial_t u(t)$ is indeed the weak divergence of $\nabla u(t)$ for a.a. $t$. We can conclude that \eqref{S8: classic sol of PME 2} holds.

Properties (1)-(3) can be established via the same argument we used for \eqref{S1: FPME}.
 
The proof of  (4) is the same as that of Theorem~\ref{S2: other properties}(III).
\end{proof}

When $m=1$, we can apply \cite[Theorem 4.1.4]{Pazy} and Proposition~\ref{S4: contraction semigroup-Lp} to attain an analogous result to Theorem~\ref{S7: Thm: strong solution m=1}.
\begin{theorem}
\label{S8: Thm: PME m=1}
Let $(\M,g)$ be  an $(n+1)$-dimensional compact conical manifold. 
Assume that $m=1$ and $1<p<\infty$. 
Then for any $T>0$ and every $u_0\in L_p(\M)$, \eqref{S8: PME} has a unique  solution $u$ on $[0,T)$ such that
$$
u\in C^1([0,T),L_p(\M))\cap C([0,T),D( \underline{\Delta}_{F,p} )).
$$
In particular, when $u_0\in L_\infty(\M)$, $u$ is a strong solution and  satisfies (1)-(4) of Theorem~\ref{Thm: Porous Medium}.
\end{theorem}



\section*{Acknowledge}

Parts of the paper were completed while the second author was an assistant professor at Georgia Southern University. 
He would like to thank the staff and faculty at Georgia Southern University for all their support. 
We also would like to express our sincere
appreciation to the anonymous reviewer for carefully reading the manuscript.


\end{document}